\documentclass[12pt]{article}
\usepackage{amssymb, verbatim, amsmath,amsthm,amsfonts}
\usepackage{color}
\usepackage[all]{xy}
\usepackage{amsthm}
\usepackage{MnSymbol}
\usepackage{pifont}
\usepackage{ifsym}
\usepackage{enumitem}
\usepackage{mathrsfs} 

\usepackage{graphicx} 
\usepackage[hidelinks]{hyperref}
\parskip \baselineskip

\newcommand{\bc}{\begin{center}}
\newcommand{\ec}{\end{center}}
\newcommand{\be}{\begin{enumerate}}
\newcommand{\ee}{\end{enumerate}}
\newcommand{\beq}{\begin{equation}}
\newcommand{\eeq}{\end{equation}}
\newcommand{\bi}{\begin{itemize}}
\newcommand{\ei}{\end{itemize}}
\newcommand{\bd}{\begin{description}}
\newcommand{\ed}{\end{description}}
\newcommand{\ba}{\begin{array}}
\newcommand{\bea}{\begin{eqnarray*}}
\newcommand{\eea}{\end{eqnarray*}}
\newcommand{\ea}{\end{array}}
\newcommand{\bt}{\begin{tabular}}
\newcommand{\et}{\end{tabular}}

\newcommand{\bmi}{\begin{minipage}}
\newcommand{\emi}{\end{minipage}}

\newcommand{\lb}{\linebreak}

\newcommand{\K}{\cal K}

\newtheorem{stel}{Theorem}[section]
\newtheorem{defin}[stel]{Definition}
\newtheorem{lemm}[stel]{Lemma}

\newtheorem{exam}[stel]{Example}
\newtheorem{rem}[stel]{Remark}
\newtheorem{corollary}[stel]{Corollary}
\newtheorem{theo}[stel]{Theorem}

\makeatletter
\newcommand{\myitem}[1]{%
\item[#1]\protected@edef\@currentlabel{#1}%
}
\makeatother
\usepackage{multicol}

\usepackage{MnSymbol,bbding,pifont}

\newcommand{\vbf}[1]{\textbf{#1}}

\makeatletter
\newcommand{\Matrix}[1]
    {\begin{pmatrix}
      \Matrix@r #1;\@bye;\Matrix@r
     \end{pmatrix}}

\def\Matrix@r #1;{\@bye #1\Matrix@z\@bye\Matrix@s #1,\@bye, }%
\def\Matrix@s #1,{#1\Matrix@t }%
\def\Matrix@t #1,{\@bye #1\Matrix@y\@bye\@firstofone {&#1}\Matrix@t}%
\def\Matrix@y #1\Matrix@t{\\ \Matrix@r }%
\def\Matrix@z #1\Matrix@r {}
\def\@bye  #1\@bye   {}

\makeatother

\makeatletter
\newsavebox\myboxA
\newsavebox\myboxB
\newlength\mylenA

\newcommand*\xoverline[2][0.75]{%
    \sbox{\myboxA}{$\m@th#2$}%
    \setbox\myboxB\null
    \ht\myboxB=\ht\myboxA%
    \dp\myboxB=\dp\myboxA%
    \wd\myboxB=#1\wd\myboxA
    \sbox\myboxB{$\m@th\overline{\copy\myboxB}$}
    \setlength\mylenA{\the\wd\myboxA}
    \addtolength\mylenA{-\the\wd\myboxB}%
    \ifdim\wd\myboxB<\wd\myboxA%
       \rlap{\hskip 0.5\mylenA\usebox\myboxB}{\usebox\myboxA}%
    \else
        \hskip -0.5\mylenA\rlap{\usebox\myboxA}{\hskip 0.5\mylenA\usebox\myboxB}%
    \fi}
\makeatother

\begin{document}

\bc {\bf\large Near-field structures on a given scalar group}\\[3mm]
{\sc Sophie Marques and Leandro Boonzaaier} 

\it\small
Department of Mathematical Sciences, 
University of Stellenbosch, 
Stellenbosch, 7600,\lb
South Africa\\
\&
NITheCS (National Institute for Theoretical and Computational Sciences), 
South Africa \\
\rm e-mail: smarques@sun.ac.za

\it
P.O. Box 1, 
Hermannsburg
KwaZulu-Natal, 3258,\lb
South Africa\\
\rm e-mail address: leandro.boonzaaier90@gmail.com
\ec

\begin{abstract} 
With this paper, we gain a better understanding of the set of near-field structures on a fixed scalar group. If we were able to describe all near-field structures on a fixed scalar group, we could describe all near-vector spaces. The near-field structures induced by isomorphisms of canonical near-vector spaces differ by quasi-multiplicative bijections while those induced by isomorphisms of near-fields differ by multiplicative bijections. This reveals one of the fundamental differences between linear algebra and near-linear algebra. We find an explicit description of all the elementary near-vector spaces. Significantly, we construct an addition $\boxplus$ on $\mathbb{Q}$ such that $(\mathbb{Q},\boxplus, \cdot)$ is isomorphic to $(\mathbb{Q}(\sqrt{-19}),+, \cdot)$. We also describe explicitly sufficient conditions for such an isomorphism to exist for more general extensions of $\mathbb{Q}$. Moreover, under extra conditions, we still describe those structures on $(\mathbb{R}, \cdot)$, and $(\mathbb{C}, \cdot)$.   \end{abstract} 



\tableofcontents

\section{Introduction}

A near-vector space is defined over a scalar group $(F, \cdot)$. That is, a monoid with elements $0, -1, 1$ carrying some type of additive structure (see Definition \ref{scalargroup} and Definition \ref{Andredefinition}). It is known that any non-zero element of a near-vector space induces an addition, $+$, such that $(F, +, \cdot)$ becomes a near-field. Moreover, any near-vector space can be decomposed into regular components (see \cite[Theorem 4.13]{Andre} or \cite[Theorem 2.4-17]{DeBruyn}) and each non-trivial regular component is isomorphic to the near-vector space $((F, +, \cdot),(F, \cdot))^I$, as described in \cite[Theorem 4.2]{Andre} or \cite[Theorem 2.4-3]{DeBruyn}, for some non-empty index set $I$ and some addition $+$ such that $(F, +, \cdot)$ is a near-field  (see \cite[Theorem 5.2]{Andre} or \cite[Theorem 2.5-2]{DeBruyn}). Therefore, understanding all the additive structures that turn a scalar group into a near-field is the key to describing all the near-vector spaces over this given scalar group. In this paper, we advance our understanding of these additive structures. We find that those structures are interesting independently of the vector spaces they generate. 

Section \ref{prelim_mat} presents some useful preliminary material that is used throughout the paper. In particular, given a scalar monoid $(F, \cdot)$, we define different types of maps of $F$ such as multiplicative, quasi-multiplicative, scalar, and $\mathbb{Z}$-maps (see Definition \ref{endomap}). We study some properties of these maps and describe multiplicative endobijections for $\mathbb{Z}$ and $\mathbb{Q}$. These maps become important examples of maps inducing near-ring addition on a scalar monoid as we will see in the rest of the paper.

In the next section, we recall how bijections naturally induce near-ring structures. This permits us to set some important notation for the paper. We also describe the different near-ring additions on a given scalar monoid inducing isomorphic near-rings (see Lemma \ref{ringisom}) and distinct multiplications on a given additive structure also inducing isomorphic near-ring structures (see Lemma \ref{multstruc}). 

Section \ref{elem_nvs} starts with a discussion whose goal is to determine the most general structure such that an elementary $R$-module, $R$, could exist. This leads to the discussion aiming to define an elementary near-vector space as the smallest building block of a near-vector space (see Definition \ref{elemen} and Theorem \ref{elemtheo}). In doing so, we obtain an important class of near-field additions on a scalar group i.e. those induced by quasi-multiplicative bijections. 

In the last section we characterize the near-field additions on a scalar group as near-field addition maps, which are endobijections of the given scalar group satisfying some specific properties (see Definition \ref{nfam}). We then define the characteristic map permitting us to see, as usual, any near-field as a vector space over $\mathbb{Q}$ or $\mathbb{F}_p$ (see Lemma \ref{rhocha} and Lemma \ref{uniqueup}). We can fully describe all the field additions on finite fields. A significant result at the end of this last section describes sufficient conditions for an addition, $\boxplus$ to induce a structure of field $(\mathbb{Q},\boxplus, \cdot)$ isomorphic to $(\mathbb{Q}, + , \cdot)$. This addition reveals number theoretic properties of the fields isomorphic to those $(\mathbb{Q},\boxplus, \cdot)$.
Inducing extra conditions on $\mathbb{R}$ and $\mathbb{C}$, we also obtain a full description of the field addition on $(\mathbb{R},\cdot) $ and $(\mathbb{C}, \cdot)$.

\section{Preliminary material}\label{prelim_mat}
The concept of a near-vector space is a generalization of a vector space. A few authors, including Beidleman \cite{Beidleman}, Karzel \cite{Karzel} and Andr\'{e} \cite{Andre}, have defined near-vector spaces. This paper will focus on Andr\'{e}'s near-vector spaces.
We begin by giving some of the definitions we will use in this paper. Throughout the paper, given $S$ a set, we write $S^{*}$ when we refer to the non-zero elements of $S$. When $S$ is endowed with a monoid structure, we write $S^{\times}$ to refer to the invertible elements of $S$ with respect to this structure. We denote $[\![1,n]\!]$ to be the set of the integer from $1$ to $n$ where $n \in \mathbb{N}$. An endomap of $F$ is a map from $F$ to $F$ and endobijection of $F$ is a bijection from $F$ to $F$. 
We denote \\ 
$\bullet$ $\mathbb{P}_{\mathbb{N}}=\{  p | p \text{ is prime natural number} \}$, \\
$\bullet$ $\mathbb{P}_{\mathbb{Z}}=\{  \eta p | p \text{ is prime natural number and } \eta \in \{ \pm 1 \} \}$, \\
$\bullet$ $\mathbb{P}_{\mathbb{Q}^+}=\{   p^{\nu} | p \text{ is prime natural number and } \nu \in \{ \pm 1 \} \}$, \\
$\bullet$  $\mathscr{P}_{\mathbb{Z}}=\{ \{ \eta_p p | p \text{ is prime natural number} \} | \eta_p \in \{ \pm 1 \} , \ p \in \mathbb{P}_{\mathbb{N}} \}$, \\
$\bullet$ $\mathscr{P}_{\mathbb{Q}}=\{ \{ \eta_p p^{\nu_p}| p \text{ is prime natural number } \} | \eta_p, \nu_p \in \{ \pm 1\}, \  p \in \mathbb{P}_{\mathbb{N}} \}$, \\

The following definition is very convenient when working with near-vector spaces. 
\begin{defin} ( \cite[Definition 1.1]{MMJ}) \label{scalargroup}
A {\sf scalar monoid} $F$ is a tuple $F=(F,\cdot,1,0,-1)$ where $(F,\cdot,1)$ is a monoid, $0, -1\in F$, $0\neq1$,  
$0\cdot \alpha=0=\alpha\cdot 0$ for all $\alpha\in F$ and $\{ \pm 1\}$ is the solution set of the equation $x^2=1$. For all $\alpha \in F$, we denote $-\alpha$ as the element $(-1)\cdot \alpha$. When $(F^{*},\cdot,1)$ is a group, we refer to $(F,\cdot,1,0,-1)$ as a {\sf scalar group}. In the rest of the paper, unless stated otherwise, we shall denote the scalar group as $(F,\cdot)$.
\end{defin}
\begin{rem} Non-trivial unitary near-rings are examples of scalar monoids and near-fields are examples of scalar groups.  
\end{rem}
With the concept of a scalar monoid, we can define the notion of an action of a scalar monoid as follows. 
\begin{defin} (\cite[Definition 1.2]{MMJ})  \label{actiondef}
Let $V=(V,+)$ be an abelian group, $(F, \cdot , 1, 0, -1)$ be a scalar monoid and $\mu : F \times V \rightarrow V $ sending $(\alpha , v)$ to $\alpha \boxdot v$ be a map. 
\begin{enumerate}
\item We refer to the map $\mu$ as an {\sf action of $F$ on $V$}. Throughout the paper, when we call an action $\mu$, we will denote the image of an element $(\alpha , v)$ of $F \times V$, via $\mu$, as $\alpha \boxdot v$.
\item We say that an action $\mu$ is a {\sf left monoid action} if for all $\alpha , \beta \in F$, $v\in V$, $\alpha \boxdot (\beta \boxdot v) = (\alpha \cdot \beta) \boxdot v$ and $1  \boxdot v = v$.
\item We say that a monoid action $\mu$ is {\sf free} if the monoid action is free. That is, for any $\alpha, \beta \in F$ and $v \in V$, $\alpha \boxdot v = \beta \boxdot v$, then $v=0 \ \textrm{or} \ \alpha = \beta $. 
 \item We say that an action $\mu$ is {\sf compatible with the $\mathbb{Z}$-structure of $V$} if 
 \begin{itemize} 
 \item $\mu$ {\sf acts by endomorphisms}. That is, for all $\alpha \in F$, $v, w \in V$, 
$\alpha \boxdot (v + w) = \alpha \boxdot v + \alpha \boxdot w$, 
\item {\sf $-1$ acts as $-id$}. That is, for all $v \in V$, $-1 \boxdot v = -v$ and 
\item {\sf $0$ acts trivially}. That is, for all $v \in V$, $0 \boxdot v = 0$.
\end{itemize}
\item We say that an action $\mu$ is a {\sf left scalar monoid action} if it is a monoid action compatible with the $\mathbb{Z}$-structure of $V$.
\end{enumerate} 
\end{defin}

We also recall the definition of a left near-vector space. 
\begin{defin}{\normalfont (\cite[Definition 4.1]{Andre}, \cite[Definition 1.3]{MMJ})} \label{Andredefinition}
A {\sf left near-vector space} is a pair $( (V, \boxplus, \boxdot), (F,\cdot))$ (when there is no confusion we will simply denote it as $(V, F)$ or $V$) where $(F,\cdot)$ is a scalar group, $(V, \boxplus)$ is an abelian additive group and $\boxdot : F \times V \rightarrow V $ is a left scalar monoid free action such that its quasi-kernel, $Q(V) = \{v \in V \mid \forall{\alpha, \beta \in F} \exists{\gamma \in F} [ \alpha \cdot v + \beta \cdot v = \gamma \cdot v ]\}$, generates $V$ seen as an additive group. We will simply say near-vector space to refer to a left near-vector space in the following. We also refer to a near-vector space over $F$ as a $F$-near-vector space. Any trivial abelian group has a near-vector space structure through the trivial action over any scalar group. We refer to such a space as {\sf a trivial near-vector space over $F$}. We denote a trivial near-vector space as $\{0\}$.
\end{defin}

The concept of regularity is a central notion in the study of near-vector spaces. 
\begin{defin}\label{def4}{\normalfont (\cite[Definition 4.7]{Andre} )} Let $(V ,F )$ be a near-vector space.
We say that two vectors, $\vbf{u}$ and $\vbf{v}$, of $Q(V)$ are {\sf compatible} if there exists $\lambda \in F^{*}$ such that $\vbf{u} \boxplus \lambda \boxdot  \vbf{v} \in Q(V)$. 
A near-vector space $(V ,F )$ is {\sf regular} if any two vectors of  $Q(V)^{*}$ are {\it compatible.} 
\end{defin}
Regular near-vector spaces behave most like traditional vector spaces. It is straightforward to see that if $Q(V) = V,$ then $V$ is regular. Andr\'{e} proved that any near-vector space can be decomposed into regular parts \cite[Theorem 4.13]{Andre} or \cite[Theorem 2.4-17]{DeBruyn}. This serves as motivation for him to refer to regular near-vector spaces as the building blocks of his theory. Moreover, each non-zero element induces an addition on the underlying scalar group endowing it with a near-field structure as stated in the next definition. 
\begin{defin}(Lemma) \label{+u} \cite[section 2]{Andre}
Let $(V , F )$ be a near-vector space, $\vbf{u} \in Q(V)^*$  and $\alpha, \beta \in F$. We denote $\alpha +_{\vbf{u}} \beta$ the unique element $\gamma$ of $F$ such that $\alpha \boxdot \vbf{u} \, \boxplus \, \beta \boxdot \vbf{u} = \gamma \boxdot \vbf{u}$.  $+_\vbf{u}$ defines a group operation on $F$. Moreover $(F ,+_\vbf{u} , \cdot)$ is a near-field.
\end{defin}

In other words, a near-vector space structure involves the existence of possibly distinct additive structures on the scalar group, each of them endowing the latter with a near-field. A scalar group $(F, \cdot)$ with no such additive structure can only have a trivial near-vector space as $F$-near-vector space. The question of when and how many additive structures exist on a scalar group, defining a near-field structure on the scalar group, becomes central to constructing near-vector spaces. The difference between fields and near-fields in terms of those additions is related to the following remark. 

\begin{rem}
Let $(V , F )$ be a near-vector space, $\vbf{u} \in Q(V)^*$ and $\lambda \in F$. The addition $+_{\lambda \vbf{u}}$ is given by $\alpha  +_{\lambda \vbf{u}}\beta = (\alpha \cdot  \lambda +_\vbf{u} \beta \cdot \lambda ) \cdot \lambda^{-1}$, for any $\alpha, \ \beta \in F$. 
\end{rem}

To simplify the discussion throughout the paper, we make the following definition. 
\begin{defin} Let $F$ be a scalar monoid (resp. group). We say that a binary operation denoted, $+$, is a {\sf left near-ring addition (resp. left near-field addition) on $F$} if $(F, + , \cdot)$ is a left near-ring (resp. left near-field).
\end{defin}

We define morphisms between near-vector spaces as follows.
\begin{defin}{\normalfont(\cite[Definition 3.2]{HowMey} )}\label{homo}
Let $((V_1, \boxplus_1 , \boxdot_1), (F_1, \cdot_1))$ and $((V_2, \boxplus_2 , \boxdot_2), (F_2, \cdot_2))$ be two near-vector spaces.
We say that a pair $(\Psi,\varphi)$ is a {\sf homomorphism} of near-vector spaces if there is an additive homomorphism $\Psi : (V_1,\boxplus_1) \rightarrow (V_2,\boxplus_2)$ and a group isomorphism
 $\varphi : (F_1^*,\cdot_1) \rightarrow (F_2^*,\cdot_2)$
such that $\Psi (\alpha \boxdot_1  \vbf{u} ) = \varphi (\alpha)\boxdot_2 \Psi (\vbf{u} )$ for all $\vbf{u} \in V_1$ and $\alpha\in F_1^*$. Two near-vector spaces $(V_1,F_1)$ and $(V_2,F_2)$ are \textsf{homomorphic} if there exists a homomorphism of near-vector spaces  between them. When $\varphi= \operatorname{Id}$, we say that $\Psi$ is a linear map. $(\Psi,\varphi)$ is said to be an \textsf{isomorphism} of left near-vector spaces if $\Psi$ and $\varphi$ are also bijective maps. When, moreover, $\varphi= \operatorname{Id}$, we say that $\Psi$ is {\sf a linear isomorphism} of left near-vector spaces.
\end{defin}

Scalar, multiplicative and quasi-multiplicative bijections will permit us to induce isomorphic near-vector space structures on a near-field $F$ over a fixed scalar group. We define them and determine some of their properties below. 

 \begin{defin} \label{endomap}
 Let $(F_1,\cdot_1,1_1,0_1,-1_1) $ and $(F_2,\cdot_2,1_2,0_2,-1_2)$ be two scalar monoids.
 \begin{enumerate}
 \item We say that a map from $(F_1,\cdot_1,1_1,0_1,-1_1)$ to $(F_2,\cdot_2,1_2,0_2,-1_2)$  is {\sf a $\mathbb{Z}$-map} if it sends $0_1$ to $0_2$ and $\sigma (-\alpha)= - \sigma(\alpha)$ for all $\alpha \in F_1$. We say that it is  {\sf a $\mathbb{Z}$-bijection} if it is a $\mathbb{Z}$-map that is also bijective. 
\item We define a {\sf multiplicative map from $(F_1,\cdot_1,1_1,0_1,-1_1)$ to $(F_2,\cdot_2,1_2,0_2,-1_2)$} to be a monoid morphism $\varphi$  from $(F_1,\cdot_1,1_1,0_1,-1_1)$ to $(F_2,\cdot_2,1_2,0_2,-1_2)$. That is, $\varphi$ is a map such that $\varphi (1_1)= 1_2$ and $\varphi (\alpha \cdot_1 \beta ) = \varphi (\alpha) \cdot_2 \varphi (\beta)$, for all $\alpha , \ \beta \in F_1$. We say that $\varphi$ is a {\sf multiplicative bijection}, if $\varphi$ is a bijection that is also a multiplicative map.
\item We define a {\sf right scalar endomap (resp. left scalar endomap)} of $F_1$ to be a endomap of $F$ $\phi$ such that $\phi(\alpha) = \alpha \cdot_1 \lambda$ (resp. $\phi(\alpha) =  \lambda \cdot_1 \alpha$), for some $\lambda \in F^*_1$. We also say that $\phi$ is {\sf the right scalar endomap associated with $\lambda$} (resp. {\sf the left scalar endomap associated with $\lambda$}). We say that $\phi$ is a {\sf right scalar endobijection (resp. left scalar endobijection)} of $F_1$, if $\phi$ is a bijection that is also a right scalar endomap (resp. left scalar endomap) of $F_1$.
 \item We define a {\sf right quasi-multiplicative map (resp. left quasi-multiplicative map)} from $(F_1,\cdot_1,1_1,0_1,-1_1)$ to $(F_2,\cdot_2,1_2,0_2,-1_2)$ to be a map $\phi$ such that $\phi(\alpha) =\varphi ( \alpha )\cdot_2 \lambda $ for all $\alpha \in F_1$ where $\varphi$ is a multiplicative map and $\lambda \in F_2^*$. We also say that $\phi$ is {\sf the right quasi-multiplicative map from $(F_1,\cdot_1,1_1,0_1,-1_1)$ to $(F_2,\cdot_2,1_2,0_2,-1_2)$ associated with $\lambda$ and $\varphi$} (resp. {\sf the left quasi-multiplicative map from $(F_1,\cdot_1,1_1,0_1,-1_1)$ to $(F_2,\cdot_2,1_2,0_2,-1_2)$ associated with $\lambda$ and $\varphi$}). We say that $\phi$ is a {\sf right quasi-multiplicative bijection (resp. left quasi-multiplicative bijection)} if it is a bijection that is also a right quasi-multiplicative map (resp. left quasi-multiplicative map).
\end{enumerate} 
\end{defin}
 
 \begin{rem} Let $(F_1,\cdot_1,1_1,0_1,-1_1)$ and $(F_2,\cdot_2,1_2,0_2,-1_2)$ be scalar monoids. 
 When $\varphi : (F_1,\cdot_1,1_1,0_1,-1_1)\rightarrow (F_2,\cdot_2,1_2,0_2,-1_2)$ is a bijection and $\varphi (\alpha \cdot_1 \beta ) = \varphi ( \alpha) \cdot_2 \varphi ( \beta)$, then $\varphi (1_1) = 1_2$. Indeed, let $\alpha \in F_1$. $\varphi ( 1_1\cdot_1 \alpha ) = \varphi( 1_1) \cdot_2 \varphi (\alpha) = \varphi(\alpha)$. Since $\varphi$ is a bijection, we obtain $ \varphi( 1_1) \cdot_2 \beta = \beta$, for all $\beta \in F_1$. Since $1_2$ is the only element for $F_2$ having this property, we obtain $ \varphi( 1_1)= 1_2$.
 \end{rem} 
 
The next lemma proves that any multiplicative bijection is a $\mathbb{Z}$-map. 

\begin{lemm} \label{Zendo} Let $(F_1,\cdot_1,1_1,0_1,-1_1)$ to $(F_2,\cdot_2,1_2,0_2,-1_2)$ be scalar monoids. 
Any multiplicative bijection $\varphi$ from $(F_1,\cdot_1,1_1,0_1,-1_1)$ to $(F_2,\cdot_2,1_2,0_2,-1_2)$ is a $\mathbb{Z}$-map such that for any $\alpha \in F_1^*$, we have $\varphi (\alpha^{-1}) = \varphi (\alpha)^{-1}$. 
\end{lemm}
\begin{proof}
Let $\varphi$ be a multiplicative bijection. By definition, we have $\varphi (1_1) =1_2$.  We also know $0_1 \cdot \alpha = 0_1$. Therefore, $\begin{array}{lrll} \varphi(0_1 \cdot_1 \alpha) &=& \varphi(0_1) \Leftrightarrow  \varphi(0_1) \cdot_2 \varphi(\alpha) = \varphi(0_1). \\
\end{array}
$ Since $\varphi$ is a bijection we obtain the $  \varphi(0_1) \cdot_2 \beta=   \varphi(0_1) $, for all $\beta \in F_1$. Since $0_2$ is the only element of $F_2$ having this property, we get $\varphi(0_1)=0_2$.
Since $(-1_1)^{2} = 1_1$, we have $\varphi ((-1_1)\cdot_1 (-1_1)) = \varphi (-1_1)\cdot_2 \varphi (-1_1) = \varphi (1_1) = 1_2$. Thus, $(\varphi (-1_1))^{2} = 1_2$, so that $\varphi (- 1_1) \in \lbrace 1_1 , - 1_1\rbrace$, since the 
solution set of the polynomial $x^{2} = 1_1$ is  $ \lbrace 1_1 , - 1_1\rbrace$, by definition of a scalar group. Since $\varphi$ is a monoid morphism, this proves that $\varphi$ is a $\mathbb{Z}$-map concluding the proof.  From the multiplicativity and the bijectivity of the map $\varphi$ it is not hard to prove that for any $\alpha \in F_{1}^{*}$, we have $\varphi (\alpha^{-1}) = \varphi (\alpha)^{-1}$.
\end{proof}
\begin{rem} 
When $(F_1, +_1, \cdot_1)$ and $(F_2, +_2, \cdot_2)$ are fields, any multiplicative bijection $\varphi$ from $(F_1, +_1, \cdot_1)$ to $(F_2, +_2, \cdot_2)$ sends a primitive $n^{th}$ root of unity in $F_1$  to a primitive $n^{th}$ root of unity in $F_2$.  To prove that any primitive $n^{th}$ root of unity is sent to a primitive $n^{th}$ root of unity via $\varphi$, since $\varphi$ is multiplicative it is enough to prove that any primitive $(p^n)^{th}$ root of unity is sent to a primitive $(p^n)^{th}$ root of unity via $\varphi$. And this can be proven with induction on $n$. 
\end{rem} 
The following Lemma establishes properties of a quasi-multiplicative map. 
\begin{lemm} \label{scalarhom} Let $(F_1,\cdot_1,1_1,0_1,-1_1)$, $(F_2,\cdot_2,1_2,0_2,-1_2)$ be scalar monoids, and $\phi$ be a quasi-multiplicative map from $(F_1,\cdot_1,1_1,0_1,-1_1)$ to $(F_2,\cdot_2,1_2,0_2,-1_2)$ associated with $\lambda \in F_2$ and a multiplicative map $\varphi$. $\phi$ is a bijection if and only if $\lambda\in F_2^\times$ and $\varphi$ is a bijection. In particular, if $\phi$ is a scalar map associated with $\lambda \in F_2$, then $\phi$ is a bijection if and only if $\lambda \in F_2^\times$. 
\end{lemm}
\begin{proof}
Let $\phi$ be a quasi-multiplicative map associated with $\lambda \in F_2^{\times}$ and $\varphi$ be a multiplicative bijection. 

Let $\alpha_{1},\alpha_{2} \in F_{1}$ and suppose $\phi(\alpha_{1}) = \phi(\alpha_{2})$. Thus, $\varphi(\alpha_{1})\cdot_{2} \lambda = \varphi(\alpha_{2}) \cdot_{2} \lambda$. Since, $\lambda \in F_2^{\times}$, we have $\varphi(\alpha_{1}) = \varphi(\alpha_{2})$. Since 
$\varphi$ is injective, we have $\alpha_{1} = \alpha_{2}$. Thus, $\phi$ is 
injective.
Consider $\gamma \in F_2$. We want to show there exists $\alpha \in F_1$ such that $\phi(\alpha) = \gamma$. However, $\phi(\alpha) = \varphi(\alpha) \cdot_2 \lambda$ and since $\lambda \in F_2^{\times}$, this is equivalent to requiring the existence of $\alpha$ such that $\varphi(\alpha) = \gamma \cdot_2 \lambda^{-1}$. But, $\varphi$ is surjective, hence such an $\alpha$ does exist. Thus, $\phi$ is surjective. We conclude that $\phi$ is a 
bijection.

Conversely, suppose that $\phi$ is a bijection. Then there is $\iota \in F_1$ such that $1_2 =\phi(\iota)= \varphi(\iota) \cdot_2 \lambda$. We first prove that $\varphi$ is one-to-one. 
Indeed, by contradiction, if $\varphi$ is not one-to-one we get $\varphi( \alpha_1 )= \varphi( \alpha_2 )$, for some $\alpha_1\neq \alpha_2$. But then $\phi( \alpha_1) =\phi( \alpha_2)$. 
Contradicting that $\phi$ is a bijection. Thus $\varphi$ is a bijection onto its image. We write $\varphi_I^{-1}$ to be the inverse of $\varphi$ co-restricted to its image.
We set $\chi: \operatorname{Im} (\varphi) \rightarrow F_1$ to be the map defined by $\chi (\alpha) = \varphi_I^{-1} (\alpha) \cdot_1 \iota$, where $\alpha \in \operatorname{Im} (\varphi)$. Then, since $\varphi$ is multiplicative, for $\alpha\in \operatorname{Im} (\varphi)$, $\phi \circ \chi (\alpha) = \phi ( \varphi^{-1} (\alpha ) \cdot_1 \iota) = \varphi (\varphi^{-1} (\alpha ) \cdot_1 \iota) \cdot_2 \lambda = \alpha \cdot_2 ( \varphi (\iota) \cdot_2 \lambda )= \alpha$. 
Since $\phi$ is a bijection, $\phi_I$ restricted to $\operatorname{Im} (\varphi)$ and co-restricted to $\operatorname{Im} (\chi)$ admits a unique inverse. Therefore, this inverse is $\chi$. As a consequence, since $\varphi$ is multiplicative, we also have for any $\alpha \in \operatorname{Im} (\varphi)$, $\chi \circ \phi ( \alpha )= \chi ( \varphi_I ( \alpha ) \cdot_2 \lambda ) =  \varphi_I^{-1} ( \varphi_I ( \alpha ) \cdot_2 \lambda ) \cdot_1 \iota = \alpha \cdot_{1} \varphi_I^{-1} (\lambda \cdot_2 \varphi_I( \iota))= \alpha$. Then, since $\varphi(1_1)=1_2$, $1_2 \in \operatorname{Im}( \varphi)$, applying the equality to $\alpha =1_1$ we get $\varphi_I^{-1} (\lambda \cdot_2 \varphi_I( \iota)) = 1_1$ and $\lambda \cdot_2 \varphi_I( \iota) = \varphi_I (1_1)=1_2$. Thus, $\lambda \cdot_2 \varphi( \iota)=1_2$ and $\varphi(\iota)$ is the right and left inverse of $\lambda$ so that $\lambda \in F_2^\times$. Finally, $\varphi$ is onto. Indeed, let $\gamma \in F_2$. Then since $\phi$ is bijective, there is $\alpha\in F_1$ such that $\gamma  \cdot_2 \lambda = \phi(\alpha)$ and $\gamma = \phi(\alpha) \cdot_2 \lambda^{-1} = \varphi (\alpha)$ proving that $\varphi$ is onto. 
\end{proof}

With the next Lemma, we learn that multiplicative bijection send inverses to inverses. 
 \begin{lemm} \label{multihom} Let $(F_1,\cdot_1,1_1,0_1,-1_1)$, $(F_2,\cdot_2,1_2,0_2,-1_2)$ be scalar monoids and \\ 
 $\varphi: (F_1,\cdot_1,1_1,0_1,-1_1)\rightarrow (F_2,\cdot_2,1_2,0_2,-1_2)$ be a multiplicative bijection and $\lambda\in F_2^\times$. Then $\varphi (F_1^\times) = F_2^\times$ and the map sending $\alpha$ to $\lambda^{-1} \cdot_2 \varphi ( \alpha) \cdot_2 \lambda$ is also multiplicative bijection.
 \end{lemm} 
  \begin{proof}
 Let $\alpha \in F_1^\times$, then there is $\beta \in F_1$ such that $\alpha \cdot_1 \beta =1_1$ and $\beta \cdot_1 \alpha=1_1$. Thus, $1_2= \varphi (1_1) = \varphi (\alpha \cdot_1 \beta) = \varphi (\alpha) \cdot_2 \varphi (\beta)$ and $1_2= \varphi (1_1) = \varphi ( \beta \cdot_1 \alpha) =  \varphi (\beta) \cdot_2 \varphi (\alpha)$. That is $\varphi (\alpha)\in F_1^\times$. The converse inclusion is proved similarly using $\varphi^{-1}$. 
 
 We prove that the map $\psi$ sending $\alpha$ to $\lambda^{-1} \cdot_2 \varphi ( \alpha) \cdot_2 \lambda$ is also multiplicative bijection. For all $\alpha,\beta \in F_1$, since $\varphi$ is multiplicative by assumption, we have 
$$\begin{array}{lll} \psi (\alpha) \cdot_2 \psi (\beta)&=& ( \lambda^{-1} \cdot_2 \varphi( \alpha) \cdot_2 \lambda) \cdot_2 ( \lambda^{-1} \cdot_2\varphi ( \beta) \cdot_2 \lambda)\\
&=& \lambda^{-1} \cdot_2 \varphi( \alpha) \cdot_2\varphi ( \beta) \cdot_2 \lambda =\lambda^{-1} \cdot_2 \varphi ( \alpha\cdot_1 \beta) \cdot_2 \lambda = \psi(  \alpha\cdot_1 \beta).   \end{array}$$
and its inverse is the map $\psi^{-1}$ sending $\alpha$ to $\lambda \cdot_2 \varphi^{-1} ( \alpha) \cdot_2 \lambda^{-1}$.

 \end{proof}
The next Lemma gives a different useful characterization for a quasi-multiplicative bijection. 
\begin{lemm} \label{qmc} Let $(F_1,\cdot_1,1_1,0_1,-1_1)$, $(F_2,\cdot_2,1_2,0_2,-1_2)$ be scalar monoids.\\ Let $\phi : (F_1,\cdot_1,1_1,0_1,-1_1)\rightarrow (F_2,\cdot_2,1_2,0_2,-1_2)$ be a map. 
The following assertions are equivalent. 
\begin{enumerate}
\item $\phi$ is a right quasi-multiplicative bijection associated with $\varphi $ a multiplicative bijection and $\lambda$; 
\item $\phi$ is a bijection such that for all $\alpha, \beta \in F_1$, $\phi ( \alpha \cdot_1 \beta) = \phi (\alpha )  \cdot_2 \phi(\gamma \cdot_1 \beta)$, the element $\gamma$ is such that $ \phi (\gamma)=1_2$ and $\lambda = \phi(1_1)$. Moreover, $\gamma$ and $\lambda$ are invertible;
 \item  $\phi$ is a bijection such that for all $\alpha, \beta \in F_1$, $\phi ( \alpha \cdot_1 \beta) = \phi (\alpha \cdot_1 \gamma ) \cdot_2   \phi(\beta)$, the element $\gamma$ is such that $ \phi (\gamma)=1_2$ and $\lambda=\phi(1_1)$. Moreover, $\gamma$ and $\lambda$ are invertible;
 \item the map $\varphi$ that sends $\alpha$ to $\phi(\alpha
\cdot_1\gamma)$ is a multiplicative bijection, the element $\gamma$ is such that $ \phi (\gamma)=1_2$ and $\lambda=\phi(1_1)$. Moreover, $\gamma$ and $\lambda$ are invertible;
 \item the map $\varphi^{\lambda}$ that sends $\alpha$ to $ \phi(\gamma \cdot_1 \alpha)$ is a multiplicative bijection, $\varphi$ is the map sending $\alpha$ to $\lambda\cdot_2 \varphi^{\lambda} (\alpha)\cdot_2 \lambda^{-1}$, the element $\gamma$ is such that $ \phi (\gamma)=1_2$ and $\lambda=\phi(1_1)$. Moreover, $\gamma$ and $\lambda$ are invertible.
 \end{enumerate}
 In particular, in 1., we have that $\varphi$ and $\lambda$ are unique.
\end{lemm} 

\begin{proof}
1. $\Rightarrow$ 2. Let $\phi : (F_1,\cdot_1,1_1,0_1,-1_1)\rightarrow (F_2,\cdot_2,1_2,0_2,-1_2)$ be a right quasi multiplicative bijection associated with $\varphi$ and $\lambda$. Let $\alpha , \beta \in F_1$. Then, 
$$\phi(\alpha \cdot_1 \beta) = \varphi(\alpha \cdot_1 \beta) \cdot_2 \lambda = \varphi(\alpha) \cdot_2 \varphi(\beta) \cdot_2 \lambda.$$ From the Lemma \ref{scalarhom} , we know that $\lambda \in F_2^\times$, therefore we can write 
$$\begin{array}{lll} \varphi(\alpha) \cdot_2 \varphi(\beta) \cdot_2 \lambda &=& \varphi(\alpha) \cdot_2 \lambda \cdot_2 \lambda^{-1} \cdot_2 \varphi(\beta) \cdot_2 \lambda = \varphi(\alpha) \cdot_2 \lambda \cdot_2 \varphi(\varphi^{-1}(\lambda^{-1})) \cdot_2 \varphi(\beta) \cdot_2 \lambda\\
&=& \varphi(\alpha) \cdot_2 \lambda \cdot_2 \varphi(\varphi^{-1}(\lambda^{-1}) \cdot_1 \beta) \cdot_2 \lambda
\end{array} .$$ Moreover, $\lambda=\phi(1_1)$. Setting $\gamma := \varphi^{-1}(\lambda^{-1})$, we observe that $\phi(\gamma) = \varphi ( \varphi^{-1}(\lambda^{-1})) \cdot_2 \lambda = 1_2$ and $\gamma$ is invertible since $\lambda$ is invertible. We finally have $\phi(\alpha \cdot_1 \beta) = \phi(\alpha) \cdot_2 \phi(\gamma \cdot_1 \beta)$. \\
2. $\Rightarrow$ 3. Suppose that $2.$ is true. Let $\alpha, \beta \in F_1$, 
$$\phi ( \alpha \cdot_1 \beta)= \phi ( (\alpha \cdot_1 \gamma) \cdot_1 (\gamma^{-1} \cdot_1 \beta))= \phi (\alpha \cdot_1 \gamma )  \cdot_2 \phi(\gamma \cdot_1 \gamma^{-1} \cdot_1 \beta) = \phi (\alpha \cdot_1 \gamma ) \cdot_2   \phi(\beta). $$
Moreover, $\varphi$ is a bijection since $\phi$ is a bijection and $\gamma$ is invertible, by assumption. 

3. $\Rightarrow$ 4. We assume 3. is true. We define the map $\varphi$ sending $\alpha$ to $\phi (\alpha \cdot_1 \gamma)$ for any $\alpha \in F$. Let $\alpha , \beta \in F_1$. Then 
$$\begin{array}{lll} \varphi(\alpha \cdot_1 \beta) &=& \phi((\alpha \cdot_1 \beta)\cdot _1 \gamma)= \phi(\alpha \cdot_1 ( \beta\cdot _1 \gamma)) \\
&=& \phi(\alpha \cdot_1 \gamma) \cdot_2 \phi( \beta\cdot _1 \gamma) = \varphi(\alpha ) \cdot_2 \varphi(\beta).\end{array} $$

4. $\Rightarrow$ 5. We assume 4. is true. 
Let $\alpha \in F_1$. We now prove that $ \phi(\gamma \cdot_1 \alpha)=\lambda^{-1}\cdot_2 \phi( \alpha \cdot_1 \gamma)\cdot_2 \lambda$. Indeed, by assumption, we have
$$\begin{array}{lll} \phi( \gamma \cdot_1 \alpha) &=& \phi ( \gamma \cdot_1 \alpha \cdot_1 \gamma^{-1} \cdot_1 \gamma) =\phi(\gamma \cdot_1 \alpha \cdot_1 \gamma ) \cdot_2 \phi (1_1)  \\
&=& \phi (\gamma  \cdot_1 \gamma) \cdot_2 \phi(\alpha \cdot_1 \gamma) \cdot_2 \phi(1_1)
\end{array}$$
Finally, 
$$\phi( \gamma \cdot_1 \gamma)  \cdot_2 \phi(1) = \phi( \gamma \cdot_1 \gamma)  \cdot_2 \phi( \gamma^{-1} \cdot_1 \gamma ) = \phi(\gamma)=1_2.$$
Since $\lambda= \phi(1_1)$ is invertible, we have $\lambda^{-1} = \phi( \gamma \cdot_1 \gamma) $, proving the result. Finally, we deduce that $\varphi^\lambda$ is a multiplicative bijection from Lemma \ref{multihom}.

5. $\Rightarrow$ 1. Suppose that $5.$ is true. 
Then for $\alpha \in F$, by Lemma \ref{Zendo} , 
$$\begin{array}{lll} \phi (\alpha ) &=& \phi ((\gamma \cdot_1 \gamma^{-1}) \cdot_1 \alpha)= \varphi^{\lambda}(  \gamma^{-1}  \cdot_1 \alpha)=   \varphi^{\lambda}(  \gamma^{-1})\cdot_2  \varphi^{\lambda}( \alpha) \\
&=& \varphi^{\lambda}(  \gamma)^{-1}  \cdot_2  \varphi^{\lambda}( \alpha) =  (\lambda^{-1}  \cdot_2 \varphi (  \gamma)^{-1}   \cdot_2 \lambda) \cdot_2  (\lambda^{-1}  \cdot_2 \varphi( \alpha) \cdot_2 \lambda)  \\
&=&  \lambda^{-1}  \cdot_2 \varphi (  \gamma)^{-1}    \cdot_2 \varphi( \alpha) \cdot_2 \lambda= \varphi( \alpha) \cdot_2 \lambda
.\end{array}$$
Since $\lambda^{-1}  \cdot_2 \varphi (  \gamma)^{-1}=1_2$. Indeed, $\lambda^{-1}  \cdot_2 \varphi (  \gamma)^{-1}  \cdot_2 \lambda = \varphi^{\lambda} (  \gamma)^{-1}= \phi(\gamma \cdot_1 \gamma)^{-1} =\lambda$. The latter follows from the equality $\lambda=\phi(\gamma \cdot_1 \gamma)^{-1}$. This equality is a consequence of
$ \phi(\gamma \cdot_1 \gamma) \cdot_2 \lambda = \phi(\gamma \cdot_1 \gamma) \cdot_2 \phi(1_1) = \phi(\gamma \cdot_1 \gamma) \cdot_2 \phi(\gamma \cdot_1 \gamma^{-1})= \phi(\gamma )=1_2$, since $\lambda$ is invertible. 
\end{proof}

 The next lemma gives the structure of the set of right quasi-multiplicative (resp. left quasi-multiplicative, resp. multiplicative, resp. left scalar, resp. right scalar, resp. $\mathbb{Z}$-) endobijections of $F$ 
 \begin{lemm}\label{eqm} Let $(F, \cdot , 1, 0, -1) $ be a scalar monoid. The set of right quasi-multiplicative (resp. left quasi-multiplicative, resp. multiplicative, resp. left scalar, resp. right scalar, resp. $\mathbb{Z}$-) endobijections of $F$ form a group under composition. 
 \end{lemm}
\begin{proof} 
Let $(F, \cdot , 1, 0, -1) $ be a scalar monoid. We prove that each of these sets are subgroups of the group of the bijections of $F$. We prove the result for quasi-multiplicative endobijections. The others are obtained similarly.
 \begin{enumerate} 
 \item The identity map is the identity of each group, since the identity map is a right quasi-multiplicative endomap.
 \item The composite of right quasi-multiplicative endobijections is a right quasi-multiplicative endobijection. Let $\phi_1$ (resp. $\phi_2$) be a quasi-multiplicative endobijection associated with the multiplicative endobijection $\varphi_1$ (resp. $\varphi_2$)  and $\lambda_1 \in F^\times$ (resp. $\lambda_2 \in F^\times$) (see Lemma  \ref{scalarhom}). Let $\alpha \in F$. Then 
$$\begin{array}{lll} (\phi_1 \circ \phi_2) (\alpha ) &=& \phi_1 ( \phi_2 (\alpha ))= \phi_1 (\varphi_2 (\alpha ) \lambda_2)\\
&=& \varphi_1 ( \varphi_2 (\alpha ) \lambda_2)\lambda_1 = (\varphi_1 \circ \varphi_2) (\alpha) \varphi_1 (\lambda_2)  \lambda_1 \
\end{array}$$
where $\varphi_1 \circ \varphi_2$ is a multiplicative right endobijection and $\varphi_1 (\lambda_2) \lambda_1 \in F^\times$ (see Lemma  \ref{multihom}), proving that $\phi_1 \circ \phi_2$ is a right quasi-multiplicative endobijection. 
\item The inverse of a right quasi-multiplicative endobijection is a right quasi-multiplicative endobijection. Let $\phi$ be a quasi-multiplicative endobijection associated with the multiplicative endobijection $\varphi$ and $\lambda \in F^\times$ (see Lemma  \ref{scalarhom}), such that $\phi(\alpha ) =  \varphi (\alpha) \lambda$. Then it is clear that $\phi$ is invertible and its inverse is $\phi^{-1} ( \alpha) =\varphi^{-1} (\alpha  \lambda^{-1})= \varphi^{-1} (\alpha) \varphi^{-1} ( \lambda^{-1})$. It is the quasi-multiplicative endobijection associated with the multiplicative endobijection $\varphi^{-1}$ and $\varphi^{-1} ( \lambda^{-1})\in F^\times$ (see Lemma  \ref{multihom}).
\end{enumerate}
\end{proof}
When $K= \mathbb{R}$ and $\mathbb{C}$, we have already computed all the continuous multiplicative endobijections of $K$ in \cite{MarquesBoonz2022A}. We recall these results with the next Definition (Lemma). In the following, $\overline{u}$ will denote the complex conjugate of a complex number $u$, and when $u$ is a real number $\overline{u}$ is simply $u$. 

\begin{defin}(Lemma)\label{R} \cite[Theorem 3.12]{MarquesBoonz2022A} The continuous multiplicative endobijections of $\mathbb{C}$ are either of the form 
$$\begin{array}{cccl}  \epsilon_\alpha : &\mathbb{C} & \rightarrow & \mathbb{C} \\ 
& z=rs & \mapsto & r^\alpha s
\end{array} 
\text{ or }
\begin{array}{cccl}  \overline{\epsilon_\alpha} : & \mathbb{C}^*  & \rightarrow & \mathbb{C}^* \\ 
& z=rs & \mapsto &r^\alpha\overline{s}
\end{array} $$
where $r \in \mathbb{R}_{>0} $, $s \in \mathbb{S}$ and $\alpha \in \mathbb{C}\backslash i \mathbb{R}.$ 
Moreover, $\overline{\epsilon_1}$ is the complex multiplication, $ \overline{\epsilon_\alpha}=\epsilon_\alpha \circ \overline{\epsilon_1}$, $ \epsilon_\alpha=\overline{\epsilon_\alpha} \circ \overline{\epsilon_1}$, $\overline{\epsilon_1}\circ \epsilon_\alpha= \overline{ \epsilon_{\overline{\alpha}}}$ and $\overline{\epsilon_1}\circ \overline{\epsilon_\alpha}=  \epsilon_{\overline{\alpha}}$.  Finally, the inverse of $\epsilon_\alpha$ (resp.  $\overline{\epsilon_\alpha}$), $\epsilon_{\alpha}^{-1}$ (resp.  $\overline{\epsilon_\alpha}^{-1} $) is $\epsilon_{\frac{1- i \operatorname{Im}(\alpha)}{\operatorname{Re}(\alpha)} }$ (resp.  $\overline{\epsilon_{\frac{ 1+ i \operatorname{Im}(\alpha)}{\operatorname{Re}(\alpha)}  }}$) and therefore $\epsilon_\alpha$ is a homeomorphism.
We obtain all the continuous multiplicative endobijections of $\mathbb{R}$ by restricting $\epsilon_\alpha$ to $\mathbb{R}$ and taking $\alpha \in \mathbb{R}^* $. 
\end{defin}

We can also give a full description of all the endobijections of $\mathbb{Z}$ and 
 $\mathbb{Q}$. 
\begin{lemm}\label{multZ}
\begin{enumerate} 
\item The multiplicative endobijections of $\mathbb{Z}$ are defined by multiplicatively extending a bijection of $\mathbb{P}_{\mathbb{N}}$ onto $\mathcal{P}$ for some $\mathcal{P} \in  \mathscr{P}_{ \mathbb{Z}}$ and sending $0$ to $0$ and $\pm 1$ to $\pm 1$. 
\item The multiplicative endobijections of $\mathbb{Q}$ are defined by multiplicatively extending a bijection of $\mathbb{P}_{\mathbb{N}}$ onto $\mathcal{P}$ for some $\mathcal{P} \in  \mathscr{P}_{ \mathbb{Q}}$ and sending $0$ to $0$ and $\pm 1$ to $\pm 1$. We have $\varphi = (-)^a$ where $a \in\{ \pm 1\}$ are the only continuous multiplicative endobijections of $\mathbb{Q}^*$ for the topology on $\mathbb{Q}$ induced by the classical topology on $\mathbb{R}$. 
\end{enumerate}
\end{lemm}
\begin{proof} 
\begin{enumerate}
\item Suppose that $\varphi$ is a multiplicative endobijection of $\mathbb{Z}$. We have $\varphi(\pm 1) = \pm \varphi(1) $ and $\varphi (- \alpha)= - \varphi ( \alpha)$, by Lemma \ref{Zendo}. Let $p\in \mathbb{P}_{\mathbb{N}}$. Then $\varphi(p) \in \mathbb{P}_{\mathbb{Z}}$. Indeed, we prove the result by contradiction. Suppose that $\varphi(p) \notin \mathbb{P}_{\mathbb{Z}}$, then $\varphi (p)= a b$ such that $a, b \notin  \{ \pm1\}$. Applying $\varphi^{-1}$ to the equality we get $ p = \varphi^{-1} ( a b)= \varphi^{-1} ( a)  \varphi^{-1} ( b)$,
 with $\varphi^{-1} ( a) , \varphi^{-1} ( b)\notin  \{ \pm1\}$, contradicting that $p$ is prime. Now, given a bijection $\sigma : \mathbb{P}_{\mathbb{N}} \rightarrow \mathcal{P}$ for some $\mathcal{P} \in  \mathscr{P}_{ \mathbb{Z}}$, let $\eta \prod_{i=1}^s p_i^{e_i} $ be an arbitrary element of $\mathbb{Z}^*$ where $\eta \in \{ \pm 1\}$ and $e_i \in \mathbb{N}$, for all $i \in \{ 1, \cdots, s\}$. We then define the map $\varphi$ to be the map given by the rule $\varphi ( \eta \prod_{i=1}^s p_i^{e_i}) = \eta \prod_{i=1}^s \sigma (p_i)^{e_i}$. Since every set $\mathcal{P}$ in $ \mathscr{P}_{ \mathbb{Z}}$ is a generating set for $\mathbb{Z}^*$, considered as a monoid, and every element of $\mathbb{Z}^*$ can be uniquely written in terms this generating set, we have that $\varphi$ is a bijection and $\varphi$ is a multiplicative automorphism of $\mathbb{Z}^*$ by definition, proving the result. 
 
 \item Since every element $\mathcal{P}$ of  $ \mathscr{P}_{ \mathbb{Q}}$ is a generating set for $\mathbb{Q}^*$, considered as a group, and every element of $\mathbb{Q}^*$ can be uniquely written in terms this generating set, we obtain the result 2, similarly as in 1. The form of the continuous multiplicative automorphisms of $\mathbb{Q}^*$ can easily be deduced from the structure of the continuous automorphisms of $\mathbb{R}^*$ (see \cite[Theorem 3.7]{MarquesBoonz2022A}).

\end{enumerate} 
\end{proof} 

\section{Structures induced by bijective maps}\label{bijective_maps}
In this section, we exhibit the structure naturally induced by bijective maps. Let $(F_1 , \cdot_1, 1_1)$ and $(F_2 , \cdot_2, 1_2)$ two scalar monoids. 
By definition $\sigma :  (F_2, \cdot_2) \rightarrow (F_1, \cdot_1)  $ is a monoid isomorphism if $\alpha \cdot_2 \beta = \sigma^{-1} ( \sigma( \alpha) \cdot_1 \sigma( \beta ))$, for all $\alpha, \ \beta \in F$. According to this definition, we define a monoid operation induced by a bijection as follows.
 \begin{defin} 
Given $(F_1, \cdot,1)$, a monoid, $F_2$ a set such that there is a bijection $\sigma: F_2 \rightarrow F_1$. We define {\sf the monoid operation induced by $\sigma$ and $(F_1, \cdot)$ on $F_2$}, denoted $\cdot_\sigma$, to be the monoid operation sending $(\alpha , \beta)$ to $\alpha \cdot_\sigma \beta = \sigma^{-1} (\sigma(\alpha) \cdot \sigma ( \beta))$. We denote $\alpha_\sigma= \sigma^{-1} (\alpha)$ where $\alpha\in F_1$. The identity element for $\cdot_\sigma$ is $1_\sigma$. When $(F^{*}_1, \cdot, 1)$ is a group, then $(F^{*}_2, \cdot_\sigma, 1_\sigma)$ is a group and the inverse of $\alpha$ is the element that we denote $\alpha^{-_\sigma 1}$ such that $\sigma(\alpha^{-_\sigma 1}) =\sigma(\alpha)^{-1}$. When $(F_1, \cdot, 1, 0,-1)$ is a scalar group, then $(F_2, \cdot_\sigma ,1_\sigma, 0_\sigma, (-1)_\sigma)$ is a scalar group. We denote $-_{\sigma} 1 =(-1)_\sigma$. 
\end{defin}

According to Lemma \ref{+u}, we know that a near-vector space structure embeds different additive structures on the same scalar group turning the scalar group into a near-field. 
In this paper, we are trying to gain a better understanding of those possible additive structures that one can put on a given scalar monoid. The next definition translates the most natural way to induce a near-ring structure from a bijection.  

 \begin{defin}(Lemma)
 Let $(R_1, +, \cdot)$ be a left near-ring, $R_2$ be a set such that there exists $\sigma :R_2 \rightarrow R_1$ a bijection. 
\begin{enumerate}[noitemsep,topsep=-8pt, parsep=0pt,partopsep=0pt]
\item We say that $(R_2 , +_\sigma,  \cdot_\sigma )$ is {\sf the near-ring structure on $R_2$ induced by $\sigma$ and $(R_1, +, \cdot)$}. 
\item  
If $\sigma : (R_2, + ) \rightarrow (R_1, +)$ is an additive isomorphism, we say that $\cdot_\sigma$ is {\sf the additive multiplication induced by $\sigma$ on the ring $R_2$.} We say that a multiplication $\odot$ is {\sf an additive multiplication on the ring $R_2$ induced by $(R_1, +, \cdot)$}, when there is an additive isomorphism $\sigma : (R_2, + ) \rightarrow (R_1, +)$ such that $\odot = \cdot_\sigma$. We note that when $R=R_2=R_1$, $\sigma: (R, \cdot)\rightarrow (R, \cdot)$ is a multiplicative bijection, we have $\cdot_\sigma= \cdot$. 
\item We say that $+_\sigma$ is {\sf the left quasi-multiplicative (resp. right quasi-multiplicative, resp. multiplicative, resp. left scalar, resp. right scalar, resp. $\mathbb{Z}$-) addition induced by $\sigma$ and $(R_1, +, \cdot)$ on $R_2$} if $\sigma$ is a left quasi-multiplicative (resp. right quasi-multiplicative, resp. multiplicative, resp. left scalar, resp. right scalar, resp. $\mathbb{Z}$-)bijection. We say that an addition $\boxplus$ is  {\sf a left quasi-multiplicative (resp. right quasi-multiplicative, resp. multiplicative, resp. left scalar, resp. right scalar, resp. $\mathbb{Z}$-) addition on $R_2$ induced by $(R_1, +, \cdot)$}, if there is a left quasi-multiplicative (resp. right quasi-multiplicative, resp. multiplicative, resp. left scalar, resp. right scalar, resp. $\mathbb{Z}$-) bijection $\sigma: R_2 \rightarrow R_1$ such that $\boxplus= +_\sigma$. We also note that when $R:=R_1=R_2$ and $\sigma$ is an additive automorphism of $(R, +)$, we have $+_\sigma = +$.  We note that when the multiplication is also distributive on the addition on the right and $\sigma$ is a left (resp. right) quasi-multiplicative bijection associated with a multiplicative bijection $\varphi$ and $\lambda \in R^\times$, we have $+_\sigma= +_\varphi$.
\item We say that $+_\sigma$ is a {\sf continuous addition} (resp. $\cdot_\sigma$ is a {\sf continuous multiplication}), when $R_1$ and $R_2$ are endowed with a topology such that the binary operation $+$ (resp. $\cdot$) is continuous with respect to this topology and $\sigma$ is also continuous with respect to these topologies. 
\end{enumerate}
\end{defin}
\begin{proof}
  We recall that $\alpha +_\sigma \beta=\sigma^{-1} (\sigma (\alpha ) +\sigma (\beta) )$ and $\alpha \cdot_\sigma \beta= \sigma^{-1} (\sigma (\alpha ) \cdot \sigma (\beta) )$, where $\alpha$, $\beta \in R_2$. Moreover, $+_\sigma$ defines an abelian group operation on $R_2$. The zero for $+_\sigma$ is the element  $0_\sigma$ and the inverse of an element $\alpha$ in $R_2$ is an element $-_\sigma  \alpha$ such that $\sigma (-_\sigma  \alpha )= -\sigma(\alpha)$. $\cdot_\sigma$ defines a group operation on $F\backslash \{ 0_\sigma \}$. The identity is $1_\sigma$ with respect to $\cdot_\sigma$ and the inverse of $\alpha$ is the element $\alpha^{-_\sigma 1}$ such that $\sigma(\alpha^{-_\sigma 1}) =\sigma(\alpha)^{-1}$ where $\alpha$ is in $R_2$. Finally, let $\alpha, \beta, \gamma \in R_2$, we have 
$$ \begin{array}{ccl} \gamma \cdot_\sigma ( \alpha +_\sigma \beta) &=& 
\sigma^{-1} ( \sigma ( \gamma) \cdot ( \sigma (\alpha ) +\sigma (\beta) ) ) \\ 
&=& \sigma^{-1} ( \sigma ( \gamma) \cdot  \sigma (\alpha ) +\sigma ( \gamma) \cdot \sigma (\beta) ) \\
&=&  \gamma \cdot_\sigma \alpha +_\sigma \gamma \cdot_\sigma \beta .
\end{array} $$  
When $\sigma$ is a multiplicative bijection, we have $\sigma^{-1} ( \sigma(\alpha) \cdot \sigma (\beta)) =  \sigma^{-1} ( \sigma(\alpha \cdot \beta)) = \alpha \cdot \beta$, for all $\alpha, \beta \in R_2$. When $\sigma$ is an additive automorphism, we have $\sigma^{-1} ( \sigma(\alpha) + \sigma (\beta)) =  \sigma^{-1} ( \sigma(\alpha + \beta)) = \alpha  + \beta$, for all $\alpha, \beta \in R_2$. When the multiplication is distributive on the right and $\sigma$ is a quasi-multiplicative bijection associated with a multiplicative bijection $\varphi$ of $R_2$ and $\lambda \in R_1^\times$, for any $\alpha, \beta  \in R_2$,
$$\varphi^{-1} ( (\varphi (\alpha) \cdot \lambda +  \varphi (\beta) \cdot \lambda) \cdot \lambda^{-1}) = \varphi^{-1} ( \varphi (\alpha)  +  \varphi (\beta) )=\alpha +_\varphi \beta. $$
  \end{proof}
Here is an interesting example to illustrate the above. 
 \begin{exam} 
 Let $\mathbb{Q}(\alpha)$ be a number field where $\alpha$ is a primitive element over $\mathbb{Q}$ of degree $n$.
There exists a bijection $\sigma$ from any number field $\mathbb{Q}$ to $\mathbb{Q}(\alpha)$. We can take a bijection from $\mathbb{Q}$ to $\mathbb{N}$ and the Cantor tuple function to construct that bijection identifying $\mathbb{Q}(\alpha)$ with $\mathbb{Q}^n$. Then we endow $\mathbb{Q}$ with the operations $x +_\sigma
 y= \sigma^{-1} (\sigma ( x) +_\alpha \sigma (y) )$ and $x \cdot_\sigma y =\sigma^{-1} ( \sigma (x) \cdot_\alpha \sigma (y))$ where $x, y \in \mathbb{Q}$, $+_\alpha$ and $\cdot_\alpha$ are, respectively, the addition and the multiplication in $\mathbb{Q} (\alpha)$. $(\mathbb{Q}, +_\sigma, \cdot_\sigma)$ is a field isomorphic to $(\mathbb{Q} (\alpha), +_\alpha , \cdot_\alpha )$. 
 \end{exam} 

The following result is not hard to prove it describes near-ring isomorphisms in terms of the notions introduced in the section. 
 \begin{lemm}\label{ringisom}
Let $(R_1, +_1, \cdot_1)$,  $(R_2, +_2, \cdot_2)$ be two left near-rings. The following assertions are equivalent.
\begin{enumerate}
\item $\varphi$ is a near-field isomorphism between $ (R_2, +_2, \cdot_2)$ and  $(R_1, +_1, \cdot_1)$ 
\item $\varphi:  (R_2, \cdot_2) \rightarrow (R_1, \cdot_1)$ is a multiplicative bijection and $\alpha +_2 \beta=\alpha {+_1}_\varphi \beta $, for all $\alpha , \beta \in R_2$. 
\item $\varphi:  (R_2, +_2) \rightarrow (R_1, +_1)$ is a additive bijection and $\alpha \cdot_2 \beta=\alpha {\cdot_1}_\varphi \beta $, for all $\alpha , \beta \in R_2$. 
\end{enumerate} 
\end{lemm} 
\begin{proof} 
$1. \Rightarrow 2.$ Suppose that $\varphi$ is a near-ring isomorphism between $(R_2, +_2, \cdot_2)$ and $(R_1, +_1, \cdot_1)$. By definition of a near-ring homomorphism, we have that $\varphi: (R_2,  \cdot_2) \rightarrow (R_1,  \cdot_1)$ is a multiplicative bijection. Since $\varphi$ is in particular an additive morphism we have $\varphi (\alpha +_2 \beta)= \varphi (\alpha) +_1 \varphi ( \beta)$, for all $\alpha , \beta\in R_2$. Therefore, $\alpha +_2  \beta = \varphi^{-1} ( \varphi (\alpha) +_1 \varphi ( \beta))$, for all $\alpha , \beta\in R_2$.

$2. \Rightarrow 3.$ Suppose $\varphi: (R_2,  \cdot_2) \rightarrow (R_1, \cdot_1)$ is a multiplicative bijection and $\alpha+_2 \beta=\alpha {+_1}_\varphi \beta $, for all $\alpha , \beta \in R_2$. Then $\varphi$ is bijective and $\varphi (\alpha +_2 \beta) = \varphi ( \alpha {+_1}_\varphi \beta )= \varphi (\alpha) +_1 \varphi ( \beta)$, for all $\alpha , \beta\in R_2$. Therefore, $\varphi$ is an additive bijection between $ (R_2, +_2) $ and  $ (R_1, +_1) $ and $\varphi ( \alpha \cdot_2 \beta ) = \varphi (\alpha) \cdot_1 \varphi (\beta)$, for all $\alpha , \beta\in R_2$. That is $ \alpha \cdot_2 \beta =  \alpha {\cdot_1}_\varphi \beta$, for all $\alpha , \beta\in R_2$.

$3. \Rightarrow 1.$ Suppose $\varphi:  (R_2, +_2) \rightarrow (R_1, +_1)$ is a additive bijection and $\alpha \cdot_2 \beta=\alpha {\cdot_1}_\varphi \beta $, for all $\alpha , \beta \in R_2$.  Then $\varphi$ is bijective  and  $\varphi (\alpha +_2 \beta) = \varphi (\alpha) +_1 \varphi ( \beta)$ and $\varphi ( \alpha \cdot_2 \beta ) = \varphi (\alpha) \cdot_1 \varphi (\beta)$, for all $\alpha , \beta\in R_2$. That is $\varphi$ is a near-ring isomorphism.
\end{proof} 
 

\section{Elementary near-vector spaces}\label{elem_nvs}
We start this section with a discussion about ``elementary module structures'' leading to multiple additions on a given ring. This discussion gives a good insight into the theory of near-vector spaces, and also serves as a motivation for why the question of looking at the different additions on a given multiplicative set is worth considering. At the end of the section, we investigate the relationship between additions on a scalar group and elementary near-vector spaces. Exploring the notion of elementary near-vector spaces, we find that any quasi-multiplicative bijection defines a new addition on a near-field.

We first observe the definition of left modules over a ring closely. The classical definition of left modules is as follows: a left module over a ring is a pair $(M, R)$ where $R$ is a ring, $M$ is a set endowed with an abelian group structure $\boxplus$ and $\boxdot$ an action of the ring $(R,+, \cdot)$ on $(M,\boxplus)$ by additive endomorphisms on the left. The goal of the discussion that follows is to consider a pair of type $(R, R)$ such that $R$ is a ``elementary module'' over $R$ but we remove all the properties not used in the definition of a module. We start by observing that in order for such a structure to exist we might not have to have $\boxplus= +$ or $\boxdot= \cdot$. Moreover, $\boxdot$ distributes on the left over the addition $\boxplus$. We finally need that the operations $\boxdot$  and $\cdot$, $\boxplus$ and $+$ to be compatible: by this we mean $(\alpha \cdot \beta ) \boxdot \gamma = \alpha \boxdot ( \beta  \boxdot \gamma)$ and $(\alpha + \beta ) \boxdot \gamma =( \alpha \boxdot  \beta) \boxplus (\alpha  \boxdot \gamma)$, for all $\alpha , \beta , \gamma \in R$. We could define the mathematical object such that $(R,R)$ is a module without the unused axioms in the definition as follows. 

A {\sf left modnear-ring}, denoted $((R ,\boxplus , \boxdot) , (R ,+, \cdot) ) $, is a set $R$ with four binary operations on $R$ that we denote $\boxplus$, $+$, $\boxdot$ and $\cdot$ such that for all $\alpha , \beta ,\gamma \in R$,
\begin{enumerate}[noitemsep,topsep=0pt]
\item $(R, \boxplus)$ and $(R, +)$ are groups;
\item $(R, \cdot, 1)$ is a monoid;
\item $(\alpha + \beta) \boxdot \gamma= \alpha \boxdot \gamma \boxplus \beta \boxdot \gamma$;
\item  $(\alpha \cdot \beta)\boxdot \gamma = \alpha \boxdot ( \beta \boxdot \gamma) $; In other words, $(R, \cdot)$ acts on $R$ on the left via $\boxdot$;
\item The map $\mu_{1}^{L}:R \rightarrow R$ that sends $\alpha$ to $\alpha \boxdot 1$ is a bijection;
\item  $ \gamma \boxdot ( \alpha\boxplus \beta ) = \gamma \boxdot \alpha \boxplus  \gamma \boxdot \beta$.
\end{enumerate}
We say that the left modnear-ring is {\sf additively abelian on the right (resp. on the left)} if $( R, +)$ is an abelian group (resp. $(R, \boxplus)$ is an abelian group).
We say that the left modnear-ring is {\sf multiplicatively abelian} if $( R, \cdot)$ is an abelian monoid.

We define a {\sf right modnear-ring} in the same way as a left modnear-ring except $3.$, $4.$, $5.$ and $6.$ now become 3.' $  \gamma \boxdot (\alpha + \beta)= \gamma \boxdot \alpha \boxplus \gamma \boxdot \beta$; 4.' $ \gamma \boxdot (\alpha \cdot \beta) = \gamma \boxdot ( \alpha \boxdot \beta)$; 5.' $, $the map $\mu_{1}^{R}: R \rightarrow R$ sending $\alpha$ to $1 \boxdot \alpha$ is a bijection and 6.' $   ( \alpha\boxplus \beta ) \boxdot\gamma  = (\alpha \boxdot \gamma) \boxplus  ( \beta \boxdot  \gamma)$.

If $R$ is a left modnear-ring, we note that $(R, +, \cdot )$ is a right near-ring. Indeed, by $1.$ in the definition, we have that $(R, +)$ is a group, by $2.$ the multiplication is associative and by $3.$ and $5.$, we have for all $\alpha, \beta, \gamma\in F$.
$$\begin{array}{lll}  ((\alpha + \beta ) \cdot \gamma) \boxdot 1 &=&  (\alpha + \beta ) \boxdot ( \gamma \boxdot 1) =  \alpha \boxdot ( \gamma \boxdot 1)  + \beta \boxdot ( \gamma \boxdot 1)\\
& =& (\alpha \cdot \gamma ) \boxdot 1 +(\beta \cdot \gamma ) \boxdot 1 = (\alpha \cdot \gamma + \beta \cdot \gamma) \boxdot 1.  \end{array}$$ Since $\mu_{1}^{L}$ is a bijection, $(\alpha + \beta ) \cdot \gamma =  \alpha \cdot \gamma + \beta \cdot \gamma$.

Then a left modnear-ring is simply a right near-ring $(R, +, \cdot)$ with a left action $\boxdot$ on the right near-ring $(R, +, \cdot)$ compatible with the operation $\boxplus$.

If $R$ is left modnear-ring, we also note that for any $\alpha \in R$, $0 \boxdot  \alpha =0$. Indeed, $0  \boxdot \alpha = (0 + 0 ) \boxdot \alpha = 0  \boxdot \alpha \boxplus  0  \boxdot \alpha$. Thus, $0  \boxdot \alpha = 0$, adding $\boxminus  (0  \boxdot \alpha)$ both sides of the equality. Also, $-1 \boxdot  \alpha =-\alpha$. Indeed, $0  \boxdot \alpha = (1 + (-1) ) \boxdot \alpha = 1  \boxdot \alpha \boxplus  (-1)  \boxdot \alpha=\alpha \boxplus  (-1)  \boxdot \alpha$. Thus, we obtain $-1  \boxdot \alpha = -\alpha$.

The motivation for the above discussion is that the near-vector space construction resembles this definition. Only the group operation $+$ is not given, although for each vector $\gamma$, an addition $+_\gamma$ arises from the definition. But those additions are not uniform across the vector space as we will explain in the following.   

Another point to note is that a left near-ring does not necessarily define a left modnear-ring. We give some examples of modnear-rings. We start with a first example.\\
 Let $(R, +, \cdot)$ be a ring and $\varphi$ is a multiplicative automorphism.
 We have $( ( R, +, \cdot^\varphi) , ( R, +_\varphi, \cdot))$ is a left modnear-ring where $\alpha  +_\varphi \beta = \varphi^{-1} ( \varphi (\alpha ) + \varphi (\beta))$ and $\alpha \cdot^\varphi \beta = \varphi (\alpha ) \cdot \beta$, for $\alpha, \beta \in R$. Indeed, since $(R, +)$ is an abelian group $(R, +_{\varphi})$ is also an abelian group (\cite[Lemma 4.1]{MarquesBoonz2022A}). Moreover, since $(R, +, \cdot)$ is a ring, we have, for any $\alpha, \beta$ and $\gamma\in R$, 
\bea
(\alpha +_{\varphi} \beta)\cdot^{\varphi} \gamma &=& (\varphi^{-1}(\varphi(\alpha) + \varphi(\beta)))\cdot^{\varphi} \gamma = \varphi(\varphi^{-1}(\varphi(\alpha) + \varphi(\beta))) \cdot \gamma \\
&=& (\varphi(\alpha) + \varphi(\beta)) \cdot \gamma = \varphi(\alpha) \cdot \gamma + \varphi(\beta) \cdot \gamma = \alpha \cdot^{\varphi} \gamma + \beta \cdot^{\varphi} \gamma;
\eea
\bea
(\alpha \cdot \beta)\cdot^{\varphi} \gamma &=& \varphi(\alpha \cdot \beta) \cdot \gamma = \varphi(\alpha) \cdot \varphi(\beta) \gamma = \alpha \cdot^{\varphi} (\varphi(\beta) \cdot \gamma) = \alpha \cdot^{\varphi} (\beta \cdot^{\varphi} \gamma);
\eea
and
\bea
\gamma \cdot^{\varphi} (\alpha + \beta) &=& \varphi(\gamma)\cdot (\alpha + \beta) = \varphi(\gamma) \cdot \alpha + \varphi(\gamma) \cdot \beta = \gamma \cdot^{\varphi} \alpha + \gamma \cdot^{\varphi} \beta.
\eea

Another example is given as follows. 
 Let $(G, +)$ and $(G, \boxplus)$ be two (abelian) groups.  We define $\mathcal{M}(G, +, \boxplus)$ to be the set of group homomorphism from $( G, +)$ to $( G, \boxplus)$,  $$\mathcal{M}(G, +, \boxplus): =Hom_{Gr} (( G, +),( G, \boxplus) ).$$ For every $f, g \in \mathcal{M}(G, +, \boxplus)$, we define $f + g$ to be a map such that $f(x) +g(x)$  and $f \boxplus g$ to be a map such that $f(x) \boxplus g(x)$, for any $x \in G$. We have $$(( \mathcal{M}(G, +, \boxplus) , \boxplus, \circ) , ( \mathcal{M}(G, +, \boxplus) , +, \circ))$$ is a right modnear-ring. Indeed, we have $( G, +)$ and $( G, \boxplus)$ are (abelian) groups. The operation $\circ$ is associative. Let $f$, $g, \ h \in \mathcal{M}(G, +, \boxplus)$.  We have $(f \boxplus g )\circ h = f \circ h \boxplus g \circ h$ and $h \circ (f + g )= f \circ h \boxplus g \circ h$. (We note that we also have $(f + g )\circ h = f \circ h + g \circ h$.) For instance, we could consider $G=\mathbb{F}_9$, $(\mathbb{F}_9, +, \cdot)$ to be a finite field with $9$ elements, the operation $+_3$ defined as $ (\alpha^3 + \beta^3 )^{1/3}$ for any $\alpha , \beta \in \mathbb{F}_9$ and $\mathcal{M}(\mathbb{F}_9,+, +_3)$ as a specific example of the above. Considering a set with different additions on it will inform us more about that set than when considering a single addition. We could even think of equipping $R$ with all  the possible additions that would turn it into a near-ring, given a certain multiplication. That is what near-vector spaces achieve for fields.   

Continuing with the same line of thought, we want to find a good definition for an ``elementary near-vector space'' over a certain scalar group $(F, \cdot)$. 
We would then want to consider a near-vector space of the form $((F, \boxplus, \boxdot), (F, \cdot))$ such that $F\boxdot 1= F$. This last condition guarantees that $F$ has dimension $1$ over $(F, \cdot)$ and thus stays elementary. We will see further in the discussion that choosing $1$ to have this property is not restrictive.

In this generality, the definition of a left near-vector space, when it is of the form $((F, \boxplus , \boxdot), (F, \cdot))$ such that $F\boxdot 1= F$, becomes 
\begin{enumerate}[noitemsep,topsep=0pt]
\item $(F, \cdot)$ is a scalar group;
\item $(F, \boxplus)$ is an abelian group;
\item we have a scalar group action of $(F, \cdot)$ on $(F, \boxplus)$ via $\boxdot$. That is, for all $\alpha, \beta, \gamma \in F$, $1 \boxdot \gamma = \gamma$, $(-1) \boxdot \gamma = -\gamma$, $0 \boxdot \gamma = 0$, $(\gamma \cdot \alpha ) \boxdot \beta = \gamma \boxdot( \alpha  \boxdot \beta )$ and $\gamma \boxdot ( \alpha  \boxplus \beta)= \gamma \boxdot  \alpha  \boxplus \gamma \boxdot \beta$;  
\item for all $\alpha, \gamma \in F$, $\alpha \boxdot \gamma =  \gamma$ implies $\alpha = 1$;
\item for all $\alpha , \beta \in F$, there is a unique $\alpha \boxplus_1 \beta \in F$ such that $(\alpha \boxplus_1 \beta) \boxdot 1 = \alpha  \boxdot 1 \boxplus \beta  \boxdot 1$. 
\end{enumerate}
From these axioms, we can deduce the other axioms of the near-vector space definition. Indeed, $4.$ is equivalent to the point-free property. Indeed, if we suppose that $4.$ is satisfied and for all $\alpha, \beta , \gamma \in F$, $\alpha \boxdot \gamma = \beta \boxdot \gamma $, then
$\beta^{-1} \boxdot ( \alpha \boxdot \gamma) = \beta^{-1} \boxdot (  \beta \boxdot \gamma )$. By $3.$, that is equivalent to $$(\beta^{-1}\cdot \alpha) \boxdot \gamma = (\beta^{-1} \cdot   \beta) \boxdot \gamma = 1 \boxdot \gamma =\gamma.$$ Finally, we obtain $ \beta^{-1}\cdot \alpha=1$. That is, $\beta = \alpha$ and we have proven the point free property. The converse is obtained similarly. 
By \cite[Theorem 2.4]{Andre} or \cite[Theorem 2.2.6]{DeBruyn}, we know that $(F, \boxplus_1 , \cdot)$ is a left near-field.

 We also note that the following statements are equivalent:
\begin{enumerate}[noitemsep,topsep=0pt]
\item $F\boxdot 1=F$; 
\item $F\boxdot \gamma =F$, for some $\gamma \in F^*$; 
\item $F\boxdot \gamma =F$, for all $\gamma \in F^*$.
\end{enumerate}
This is a consequence of property $3.$ Indeed, $1. \Rightarrow 2.$ suppose $F\boxdot 1=F$. Let $\gamma \in F^*$. 
We have $(F \boxdot \gamma) \boxdot 1= (F \cdot \gamma) \boxdot 1= F \boxdot 1$. This implies $F \boxdot \gamma = F$ by the point-free property.
Now, for $2. \Rightarrow 3.$, suppose $F\boxdot \gamma =F$, for some $\gamma \in F^*$. Let $\alpha \in F^*$. We have $(F \boxdot \alpha) \boxdot \gamma= (F \cdot \alpha) \boxdot \gamma= F \boxdot \gamma$.  This implies $F \boxdot \alpha = F$ by the point-free property. Finally, $3. \Rightarrow 1.$ is trivial.\\
This guarantees that choosing $1$ for the equality $F\boxdot 1=F$ to be satisfied leads to no loss of generality. 

We can see that the property $F\boxdot 1=F$ and the point free property are equivalent to saying that the canonical map from $F$ to $F\boxdot 1$ sending $\alpha$ to $\alpha \boxdot 1$ is bijective. 

We note that given a left near-vector space of the form $((F, \boxplus , \boxdot), (F, \cdot))$ such that $F\boxdot 1= F$ satisfying the properties $1.$ to $5.$,  $((F, \boxplus , \boxdot), (F, \boxplus_1, \cdot))$ is very close to a modnear-ring. The difference is that each $\gamma$ on the right could define different additions. More precisely, $Q(F)=F$ and the addition $+_{\gamma \boxdot 1} $ is given by $\alpha +_{\gamma  \boxdot 1} \beta  = (\alpha \cdot \gamma +_1 \beta  \cdot \gamma ) \cdot  \gamma^{-1} $, for all $\alpha , \beta\in F$.  Indeed, we have 
$$\begin{array}{lll} \alpha \boxdot ( \gamma\boxdot 1)  +  \beta \boxdot ( \gamma \boxdot 1 ) &=&  (\alpha \cdot  \gamma)\boxdot 1  +  (\beta \cdot  \gamma) \boxdot 1 =  (\alpha \cdot  \gamma +_1  \beta  \cdot  \gamma ) \boxdot 1 \\
&=& ((\alpha  \cdot  \gamma +_1  \beta \cdot  \gamma ) \cdot  \gamma^{-1} \cdot  \gamma)  \boxdot 1\\ 
&=&  (\alpha  \cdot  \gamma +_1  \beta  \cdot  \gamma )  \cdot \gamma^{-1} \boxdot (\gamma  \boxdot 1) = ( \alpha +_{\gamma  \boxdot 1} \beta) \boxdot (\gamma  \boxdot 1).\end{array}$$ 
However, when $F$ is distributive on the right with respect to $\cdot$ onto $+_1$, then \\
$((F, \boxplus , \boxdot), (F, \boxplus_1, \cdot))$ is simply a modnear-ring.

We can associate the left near-vector space $(( F, \boxplus_1 , \cdot) , ( F, \cdot))$ with the left near-field $(F, \boxplus_1 , \cdot)$. It is relatively easy to see that $(( F, \boxplus_1 , \cdot) , ( F, \cdot))$ is a left near-vector space. In this case, the scalar group action on $(F,\boxplus_{1})$ is 
simply via $\cdot$. The properties $1 \alpha = \alpha, (-1)\alpha = -\alpha, 0  \alpha = 0$ and $(\gamma \alpha) \beta = \gamma (\alpha \beta)$, for all $\alpha, \beta, \gamma \in F$, follow immediately, as does the fixed point property. We only need to show that $\gamma (\alpha \boxplus_{1} \beta) = \gamma \alpha \boxplus_{1} \gamma \beta$, for all $\alpha, \beta, \gamma \in F$. However, this follows from the fact that $(F, \boxplus_1 , \cdot)$ is a left near-field. 
Moreover, ${\boxplus_1}_{\gamma}= \boxplus_{\gamma \boxdot 1}$, for all $\gamma \in F^*$. We have a natural isomorphism $(\Upsilon , id)$ as left near-vector spaces from $(( F, \boxplus_1 , \cdot) , ( F, \cdot))$ to $((F, \boxplus, \boxdot), (F, \cdot))$ where $\Upsilon : (F, \boxplus_1)  \rightarrow (F, \boxplus)$ sends $ \alpha$ to $\alpha \boxdot 1$. By definition of $\boxplus_1$, we have for all $\alpha, \beta\in F$, 
$$\Upsilon ( \alpha \boxplus_1 \beta )  =( \alpha \boxplus_1 \beta ) \boxdot 1= \alpha \boxdot 1 \boxplus \beta \boxdot 1 = \Upsilon (\alpha ) \boxplus \Upsilon (\beta )$$ 
and
$$ \Upsilon (\alpha \cdot \beta ) = ( \alpha \cdot \beta )\boxdot 1 = \alpha \boxdot ( \beta \boxdot 1)= \alpha \boxdot \Upsilon (\beta).$$

From the discussion above we have seen how to associate a near-vector space with a near-field and vice versa. This leads to the following definition.
\begin{defin} \begin{enumerate}
\item Given $((F, \boxplus, \boxdot), (F, \cdot))$ a left near-vector space such that $F\boxdot 1= F$.  
We refer to $(F, \boxplus_\gamma, \cdot)$ as the left near-field associated with $((F, \boxplus , \boxdot), (F, \cdot))$ at $\gamma$, where $\gamma \in F^*$. We have  $\alpha \boxplus_{\gamma  \boxdot 1} \beta  = (\alpha \cdot \gamma \boxplus_1 \beta \cdot \gamma ) \cdot \gamma^{-1} $, for all $\alpha , \beta\in F$ and $\gamma \in F^*$. 
\item Given $(F, +, \cdot)$ is a near-field. We refer to $((F, +, \cdot), (F, \cdot))$ as the canonical near-vector space associated with $(F, +,\cdot)$. 
\end{enumerate}
\end{defin} 
From the above discussion, we obtain the following Lemma. 
\begin{lemm}
Given $((F, \boxplus , \boxdot), (F, \cdot))$, a near-vector space such that $F\boxdot 1= F$. We have $(F, \boxplus_\gamma, \cdot)$ is a near-field and $((F, \boxplus, \boxdot), (F, \cdot))$ is isomorphic to $(( F, \boxplus_\gamma , \cdot) , ( F, \cdot))$ as near-vector spaces, for all $\gamma \in F^*$. 
\end{lemm} 
The next result compares the near-field structures when the canonical near-vector spaces associated with them are isomorphic. We find that those near-fields are not necessarily isomorphic as near-fields. This attests to the fact that even elementary near-vector spaces are more than just the near-fields they induce. 
\begin{lemm}\label{isomorph_add}
Let $(F_1, +_1, \cdot_1)$ and $(F_2, +_2, \cdot_2)$ be two near-fields (with the same underlying scalar group $(F, \cdot)$). 
\begin{enumerate}[noitemsep,topsep=-8pt, parsep=0pt,partopsep=0pt]
\item  $(\Psi,\varphi)$ is an isomorphism of left near-vector spaces between the canonical near-vector spaces $((F_2, +_2, \cdot_2 ),(F_2, \cdot_2) )$ and $((F_1, +_1 , \cdot_1 ), (F_1, \cdot_1))$ if and only if $\varphi$ is multiplicative bijection from $(F_2, \cdot_2)$ to $(F_1, \cdot_1)$ and we have $\alpha +_2 \beta= \alpha +_\phi \beta$, for all $\alpha , \beta \in F_2$ where $\phi$ is the quasi-multiplicative map associated with $\varphi$ and $\Psi(1) \in F_{1}^*$. 
\item  $\Psi$ is a linear-isomorphism of left near-vector spaces between the canonical near-vector spaces $((F, +,\cdot ), (F, \cdot))$ and $((F, \boxplus,\cdot ), (F, \cdot))$ if and only if we have $\alpha \boxplus \beta= \alpha +_\phi \beta$, for all $\alpha , \beta \in F$ and $\phi$ is the scalar map associated with  $\Psi(1) \in F^*$.
\item  $(\Psi,\varphi)$ is an isomorphism of left near-vector spaces between the canonical near-vector spaces $((F_2, +_2, \cdot_2 ),(F_2, \cdot_2) )$ and $((F_1, +_1 , \boxdot ), (F_1, \cdot_1))$ if and only if $\varphi$ is a multiplicative bijection from $(F_2,\cdot_2)$ to $(F_1,\cdot_1)$, $\Psi$ is an additive bijection from $(F_2,+_2)$ to $(F_1, +_1)$ and $\alpha \boxdot \beta= \Psi^{-1} (\varphi(\alpha) \cdot_1 \Psi(\beta ) ) $, for all $\alpha , \beta \in F$.
\end{enumerate}
\end{lemm}
\begin{proof}
\begin{enumerate}[noitemsep,topsep=-8pt, parsep=0pt,partopsep=0pt]
\item  Let $(\Psi,\varphi )$ be an isomorphism of left near-vector spaces between the canonical vector spaces $((F_2, +_2, \cdot_2), (F_2, \cdot_2))$ and $((F_1, +_1 , \cdot_1 )$, $(F_1, \cdot_1))$. By definition of isomorphism of near-vector spaces, we have for any $\alpha,\beta\in F_2$, $\Psi( \alpha \cdot_2 \beta) = \varphi  (\alpha) \cdot_1 \Psi (\beta )$ where $\varphi $ is a multiplicative bijection from $(F_2, \cdot_2)$ to $(F_1, \cdot_1)$ and $\Psi$ is an additive bijection from $(F_2, +_2)$ to $(F_1, +_1)$. Therefore, since we can write $\alpha = \alpha \cdot_2 1_2$, we obtain $\Psi ( \alpha ) = \Psi (\alpha \cdot_2 1_2)= \varphi  (\alpha) \cdot_1 \Psi (1_1)$. We set $\lambda = \Psi(1_1)$. We have $\lambda \in F_{1}^*$, since $\Psi(0_2)= 0_1$ and $\Psi$ is a bijection. Therefore, the inverse of $\Psi$ is given by $\Psi^{-1} (\alpha) = \varphi ^{-1} (\alpha  \cdot_2 \lambda^{-1})$. Since $\Psi$ is in particular an additive bijection we have $\Psi (\alpha +_2 \beta)= \Psi (\alpha) +_1 \Psi ( \beta)$. That is, $\alpha +_2 \beta =\Psi^{-1} (  \Psi (\alpha) +_1 \Psi ( \beta)) =\varphi ^{-1}((\varphi  (\alpha) \cdot_1 \lambda +_1 \varphi  (\beta) \cdot_1 \lambda) \cdot_1 \lambda^{-1})$. For the converse, we suppose that $\varphi $ is multiplicative bijection from $(F_2,\cdot_2)$ to $(F_1,\cdot_1)$, $\lambda= \Psi(1_2) \in F_{1}^{*}$ and 
$\alpha +_2  \beta =\Psi^{-1} (  \Psi (\alpha) +_1 \Psi ( \beta)) = \varphi ^{-1}(( \varphi (\alpha) \cdot_1 \lambda +_1 \varphi  (\beta) \cdot_1 \lambda) \cdot_1 \lambda^{-1})$.  We define $\Psi (\alpha ) = \varphi (\alpha) \cdot_1 \lambda$, for any $\alpha \in F_2$. We prove that $\Psi$ is an isomorphism of left near-vector spaces between $((F_2, +_2, \cdot_2 ), (F_2, \cdot_2))$ and $((F_1, +_1, \cdot_1 ), (F_1, \cdot_1))$.  We have that $\Psi$ is a bijection. Indeed, the inverse map $\Psi^{-1}$ of $\Psi$ is defined by $\Psi^{-1} (\alpha) = \varphi ^{-1} (\alpha \cdot_1 \lambda^{-1})$, for all $\alpha \in F_2$. Moreover, for all $\alpha , \beta \in F_2$, 
$$\Psi ( \alpha +_2 \beta ) = \Psi (\varphi ^{-1}(( \varphi  (\alpha) \cdot_1  \lambda +_1 \varphi  (\beta)  \cdot_1 \lambda) \cdot_1 \lambda^{-1}) )= \varphi  (\alpha)  \cdot_1 \lambda +_1 \varphi  (\beta) \cdot_1 \lambda= \Psi (\alpha ) +_1 \Psi (\beta),$$ 
and
$$\Psi (\alpha \cdot_2 \beta ) = \varphi (\alpha  \cdot_2 \beta)  \cdot_1 \lambda = (\varphi ( \alpha)  \cdot_1 \varphi (\beta) )\cdot_1 \lambda = \varphi ( \alpha)  \cdot_1 (\varphi (\beta) \cdot_1 \lambda) = \varphi (\alpha ) \cdot_1 \Psi (\beta).$$
Finally, since by assumption, $\varphi $ is a multiplicative bijection from $(F_2,\cdot_2)$ to $(F_1,\cdot_1)$, we conclude the proof. 
\item This is a particular case of $3.$ where $\varphi=id$. For this reason, we only prove $3.$. 
\item  Let $(\Psi,\varphi )$ be an isomorphism of left near-vector spaces between $((F_2, +_2, \cdot_2 ), (F_2, \cdot_2))$ and $((F, +_1 , \boxdot )$, $(F_1, \cdot_1))$. We have for any $\alpha,\beta\in F$, $\Psi( \alpha \boxdot \beta) = \varphi  (\alpha) \cdot_1 \Psi (\beta )$ where $\varphi $ is a multiplicative bijection from $(F_2,\cdot_2)$ to $(F_1,\cdot_1)$ and $\Psi$ is an additive bijection from $(F_2,+_2)$ to $(F_1, +_1)$. Therefore, $ \alpha \boxdot \beta =\Psi^{-1} ( \varphi  (\alpha) \cdot_1 \Psi (\beta ))$, for any $\alpha,\beta\in F_2$. 
\end{enumerate}
\end{proof}

\begin{corollary}\label{dvr}
 Let $(F_1, +_1, \cdot_1)$ and $(F_2, +_2, \cdot_2)$ be two left near-fields and $(F_1, +_1, \cdot_1)$ be a division ring. The canonical near-vector space $((F_1, +_1, \cdot_1 ), (F_1, \cdot_1))$ is isomorphic to the canonical near-vector space $((F_2, +_2 , \cdot_2 ), (F_2, \cdot_2))$ if and only if $ (F_1, +_1, \cdot_1)$ is isomorphic to $(F_2, +_2, \cdot_2)$ as a near-field. In particular, when the canonical near-vector space $((F_1, +_1, \cdot_1 ), (F_1, \cdot_1))$ is isomorphic to the canonical near-vector space $((F_2, +_2, \cdot_2 ), (F_2, \cdot_2))$, then $(F_2, +_2, \cdot_2)$ is a division ring.
\end{corollary} 
\begin{proof}
 Suppose that the canonical near-vector space $((F_1, +_1, \cdot_1), (F_1, \cdot_1))$ is isomorphic to the canonical near-vector space $((F_2, +_2 , \cdot_2), (F_2, \cdot_2))$ as a near-vector space. By Lemma \ref{isomorph_add} 3., there is a $\varphi$, a multiplicative bijection from $(F_2, \cdot_2)$ to $(F_1, \cdot_1)$, $\lambda= \Psi(1) \in F_{1}^{*}$ and $\alpha +_2 \beta= \varphi^{-1} ((\varphi(\alpha) \cdot_1 \lambda +_1 \varphi (\beta ) \cdot_1 \lambda) \cdot_1 \lambda^{-1} ) $. Since $(F_1, +_1, \cdot_1)$ is a division ring, the multiplication is distributive onto the addition so that $\alpha +_2 \beta= \varphi^{-1} ((\varphi(\alpha) +_1 \varphi(\beta )) = \alpha +_\varphi \beta$. By Lemma \ref{ringisom} 1., that proves that $ (F_1, +_1, \cdot_1)$ is isomorphic to $(F_2, +_2, \cdot_2)$ as a near-field. As a consequence, $(F_2, +_2, \cdot_2)$ is itself a division ring. 
\end{proof} 

From the previous lemma, we extract the following definition. 
\begin{defin}Let $(F_1, +_1, \cdot_1)$ and $(F_2, +_2, \cdot_2)$ be near-fields, $\phi$ be a quasi-multiplicative bijection from $(F_2, \cdot_2)$ to $(F_1, \cdot_1)$, $\varphi$ be a multiplicative bijection from $(F_2, \cdot_2)$ to $(F_1, \cdot_1)$ and $\Psi$ be a $\mathbb{Z}$-bijection $(F_2, +_2)$ to $(F_1, +_1)$. We define the scalar multiplication as the binary operation $\cdot^{\varphi, \Psi}$ sending $( \alpha, \beta)$ to $\alpha \cdot_{\varphi, \Psi} \beta =\Psi^{-1} ( \varphi (\alpha ) \cdot_1 \Psi (\beta))$, for all $\alpha, \beta \in F$.
\end{defin}
We now define the concept of an elementary near-vector space. 
\begin{defin}\label{elemen}Let $(F, +, \cdot)$ be a near-field. An {\sf elementary near-vector space} associated with $(F, +, \cdot)$ is a near-vector space of the form $((K, \boxplus , \boxdot )$, $(K, \cdot))$ that is isomorphic to the canonical near-vector space, $((F, +, \cdot), (F,\cdot))$ associated with $(F, +, \cdot)$.
\end{defin}
\begin{lemm}\label{elementary_1}Let $(F, +, \cdot)$ be a near-field and $((K, \boxplus , \boxdot )$, $(K, \cdot))$ an elementary near-vector space associated with $(F, +, \cdot)$. Then $\boxplus_{1} = +_{\phi}$ and $(K, +_{\phi}, \cdot)$ is a near-field where $\phi$ is a quasi-multiplicative bijection from $(F, \cdot )$ to $(K, \cdot)$. 
\end{lemm}
\begin{proof} We prove that $((K, \boxplus_{1} , \cdot )$, $(K, \cdot))$ is isomorphic to $((K, \boxplus , \boxdot )$, $(K, \cdot))$. We show that $(\K, \operatorname{Id})$ is an isomorphism from $((K, \boxplus_{1} , \cdot )$, $(K, \cdot))$ to $((K, \boxplus , \boxdot )$, $(K, \cdot))$, where $\K$ is the map sending $\alpha$ to $\alpha \boxdot 1$. We first prove that $\K$ is a bijection. We have that $((K, \boxplus , \boxdot )$, $(K, \cdot))$ is isomorphic to $((F, +, \cdot), (F,\cdot))$. That is there is a near-vector space isomorphism $(\Psi, \varphi)$ such that $\alpha \boxdot 1  = \Psi^{-1} ( \varphi ( \alpha) \Psi(1)) $ and $\Psi(1) \in F^*$, for all $\alpha \in K$, by Lemma \ref{isomorph_add} 3. So that ${\K}^{-1}$ is defined for all $\alpha \in K$ by ${\K}^{-1} (\alpha) = \varphi^{-1} (\Psi(\alpha) \Psi(1)^{-1} )$. In particular, $\K$ is a bijection. Moreover, let $\alpha, \beta \in K$, then $${\K} (\alpha \boxplus_{1} \beta) = (\alpha \boxplus_{1} \beta) \boxdot 1 = \alpha \boxdot 1 \boxplus \beta \boxdot 1 =\K (\alpha) \boxplus \K (\beta)$$ and 
$${\K} (\alpha  \cdot \beta) = (\alpha   \cdot  \beta) \boxdot 1 = \alpha \boxdot (\beta \boxdot 1) = \alpha \boxdot \K (\beta).$$ Thus, ${\K}$ is an isomorphism. Then, from the definition of an elementary near-vector space, $((K, \boxplus_{1} , \cdot )$, $(K, \cdot))$ is isomorphic to $((F, + , \cdot )$, $(F, \cdot))$. Therefore, by Lemma \ref{isomorph_add}, $\boxplus_{1} = +_{\phi}$, where $\phi$ is a quasi-multiplicative bijection from $F$ to $K$. By \cite[Theorem 2.4]{Andre} or \cite[Theorem 2.2.6]{DeBruyn}, we know that $(K, \boxplus_{1} , \cdot)$ is a near-field.
\end{proof}

We can now describe all the elementary near-vector spaces associated with a near-field. 
\begin{theo}\label{elemtheo}Let $(F, +, \cdot)$ be a near-field. Then $((K, \boxplus , \boxdot )$, $(K, \odot))$ is an
elementary near-vector space associated with $(F, +, \cdot)$ if and only if $\boxplus = +_{\Psi}$ and
$\boxdot = \cdot^{\varphi, \Psi}$, such that $\Psi: K \rightarrow F$ is $\mathbb{Z}$-bijection and $\varphi: (K, \odot) \rightarrow (F,\cdot)$ is a multiplicative bijection. Moreover, $\boxplus_{1_K} = +_{\phi}$ 
where $\phi: K \rightarrow F$ is the quasi-multiplicative bijection associated with $\varphi$ and $\Psi(1_K)$. When $\phi$ is multiplicative bijection, $( K , +_\phi, \cdot)$ is a near-field isomorphic to $(F, +, \cdot)$.
\end{theo}
\begin{proof}
Given $(F, +, \cdot)$, a near-field. Let $((K, \boxplus , \boxdot )$, $(K, \cdot))$ be an elementary 
near-vector space associated with $(F, +, \cdot)$. By definition there is a near-vector space isomorphism ($\Psi, \varphi$) from $((K, \boxplus , \boxdot )$, $(K, \odot))$ to $((F, +, \cdot), (F,\cdot))$, where $\varphi: (K, \odot) \rightarrow (F,\cdot)$ is a multiplicative bijection. Let $\alpha, \beta\in K$. We have $\Psi(\alpha \boxplus \beta) = \Psi(\alpha) + \Psi(\beta)$, for every $\alpha, \beta \in K$.  Therefore, $\alpha \boxplus \beta = \Psi^{-1}(\Psi(\alpha) + \Psi(\beta))$ and $\Psi$ is a bijection. Furthermore, $\Psi(\alpha \boxdot \beta) = \varphi(\alpha) \cdot \Psi(\beta)$. That is, $\alpha \boxdot \beta = \Psi^{-1}(\varphi(\alpha) \cdot \Psi(\beta))$.
By definition of a near-vector space we have that for $\alpha \in K$ 
$$ \begin{array}{lrll} & -1_K \boxdot \alpha = -\alpha 
 \Leftrightarrow  \Psi^{-1}(\varphi(1_K ) \cdot \Psi(\alpha)) = \Psi^{-1}(-\Psi(\alpha)) = -\alpha \Leftrightarrow  \Psi(-\alpha) = -\Psi(\alpha), \\
\end{array}$$
and
$$ \begin{array}{lrll} 
 & 0_K \boxdot \alpha = 0_K \Leftrightarrow  \Psi^{-1}(\varphi(0_K) \cdot \Psi(\alpha)) = \Psi^{-1}(0_F)= 0_K \Leftrightarrow  \Psi(0_K) = 0_F. 
\end{array}$$

We now prove that $\boxplus_{1} = +_{\phi}$ where $\phi$ is the quasi-multiplicative bijection associated with 
$\varphi$ and $\Psi(1_K)$. Let $\alpha, \beta \in K$. We have
\begin{eqnarray}
	\alpha \boxdot 1_K \boxplus \beta \boxdot 1_K &=& \Psi^{-1}(\varphi (\alpha) \cdot \Psi(1_K)) \boxplus \Psi^{-1}(\varphi (\beta) \cdot \Psi(1_K)) \nonumber \\
	&=& \Psi^{-1}(\varphi(\alpha) \cdot \Psi(1_K) + \varphi(\beta) \cdot \Psi(1_K)). \label{envs_1}
\end{eqnarray}
However, by Lemma \ref{elementary_1}
\begin{eqnarray}
\alpha \boxdot 1_K \boxplus \beta \boxdot 1_K &=& (\alpha \boxplus_{1_K} \beta) \boxdot 1_K = (\alpha +_{\phi} \beta) \boxdot 1_K \nonumber \\
&=& \Psi^{-1}(\varphi(\phi^{-1}(\phi(\alpha) + \phi(\beta)))\cdot \Psi(1_K)). \label{envs_2}
\end{eqnarray}
From (\ref{envs_1}) and (\ref{envs_2}) we have 
\bea
\varphi(\alpha)\cdot \Psi(1_K) + \varphi(\beta)\cdot \Psi(1_K) = \varphi(\phi^{-1}(\phi(\alpha) + \phi(\beta)))\cdot \Psi(1_K)\\
(\varphi(\alpha)\cdot \Psi(1_K) + \varphi(\beta)\cdot \Psi(1_K)) \Psi(1_K)^{-1} = \varphi(\phi^{-1}(\phi(\alpha) + \phi(\beta))).
\eea
Applying $\varphi^{-1}$ to both sides of the equation, we have 
$$\phi^{-1}(\phi(\alpha) + \phi(\beta))= \varphi^{-1}((\varphi(\alpha)\cdot \Psi(1_K) + \varphi(\beta)\cdot \Psi(1_K)) \Psi(1_K)^{-1}) $$
where $\phi: K \rightarrow F$ is the quasi-multiplicative bijection associated with $\varphi$ and $\Psi(1_K)$.

Conversely, let $\Psi: K \rightarrow F$ be $\mathbb{Z}$-bijection, such that  $\boxplus= +_\Psi$ 
and 
$\varphi : (K, \odot) \rightarrow (F,\cdot)$ is a multiplicative bijection. We prove that $((K, +_{\Psi} , \cdot^{\varphi,\Psi} )$, $(K, \odot))$ is an elementary near-vector space associated with $(F,+,\cdot)$. Let $\alpha, \beta, \gamma \in K$. We have
\bea
\alpha \boxdot (\beta \boxdot \gamma) &=& \alpha \boxdot \Psi^{-1}(\varphi(\beta)\cdot \Psi(\gamma)) = \Psi^{-1}(\varphi(\alpha) \cdot \varphi(\beta) \cdot  \Psi(\gamma)) \\
	                              &=& \Psi^{-1}(\varphi(\alpha \odot \beta) \cdot  \Psi(\gamma)) = (\alpha \odot  \beta) \boxdot \gamma.
\eea
Moreover, 
\bea
\gamma \boxdot (\alpha \boxplus \beta) &=& \Psi^{-1}(\varphi(\gamma) \cdot \Psi(\alpha \boxplus \beta)) = \Psi^{-1}(\varphi(\gamma) \cdot \Psi(\alpha) + \varphi(\gamma) \cdot \Psi(\beta)).
\eea
However, 
\bea
\gamma \boxdot \alpha \boxplus \gamma \boxdot \beta &=& \Psi^{-1}(\varphi(\gamma)\cdot \Psi(\alpha)) +_{\Psi} \Psi^{-1}(\varphi(\gamma)\cdot \Psi(\beta)) \\
												&=& \Psi^{-1}(\varphi(\gamma) \cdot \Psi(\alpha) + \varphi(\gamma)\cdot \Psi(\beta)),
\eea
so that $\gamma \boxdot (\alpha \boxplus \beta) = \gamma \boxdot \alpha \boxplus \gamma \boxdot \beta$.

We want to prove $\alpha \boxdot \gamma = \beta \boxdot \gamma$ implies $\gamma = 0_K$ or $\alpha = \beta$. 
Thus, 
\bea
 & \Psi^{-1}(\varphi(\alpha) \cdot \Psi(\gamma)) &= \Psi^{-1}(\varphi(\beta) \cdot \Psi(\gamma)) \\
\Leftrightarrow &		  \varphi(\alpha) \cdot \Psi(\gamma)  &= \varphi(\beta) \cdot \Psi(\gamma).
\eea
Therefore, $\Psi(\gamma) = 0_F$ or $\varphi(\beta) = \varphi(\alpha)$. Equivalently, $\gamma = 0_K$ or $\beta = \alpha$ since $(F, +, \cdot)$ is a near-field, $\Psi: K \rightarrow F $ and $\varphi : K \rightarrow F$ are bijection.

Let $\alpha \in F$. We have 
$$ 1_K \boxdot \alpha = \Psi^{-1}(\varphi(1_K) \cdot \Psi(\alpha) ) =  \Psi^{-1}( \Psi(\alpha) ) = \alpha,$$
$$ -1_K \boxdot \alpha = \Psi^{-1}(\varphi(-1_K) \cdot \Psi(\alpha) ) =  \Psi^{-1}(- \Psi(\alpha) ) =- \alpha,$$
and
$$ 0_K \boxdot \alpha = \Psi^{-1}(\varphi(0_K) \cdot \Psi(\alpha) ) =  \Psi^{-1}(0_K) =0_K.$$
This proves that $((K, +_{\Psi} , \cdot^{\varphi,\Psi} )$, $(K, \cdot))$ is a near-vector space. To prove that $((K, +_{\Psi} , \cdot^{\varphi,\Psi} )$, $(K, \cdot))$ is an
elementary near-vector space associated with $(F, +, \cdot)$, it suffice to note that $(\Psi, \varphi)$ is an isomorphism from $((K, +_{\Psi} , \cdot^{\varphi,\Psi} )$, $(K, \cdot))$ to $((F, + , \cdot )$, $(F, \cdot))$.
\end{proof}

The following example proves that $\mathbb{Q}$ can be endowed with an addition, $\boxplus$, such that $(\mathbb{Q} , \boxplus, \cdot)$ is not isomorphic to $(\mathbb{Q} , +, \cdot)$. 

\begin{exam} \label{exq} Let $K$ be a number field such that it ring of integers is a PID with group of unit $\{ \pm 1\}$. For instance, $K = \mathbb{Q}( \sqrt{-19})$ with $\mathcal{O}_K= \mathbb{Z}\left[ \frac{ 1+ \sqrt{-19}}{2}\right]$. Since $\mathcal{O}_K$ is a PID, $\mathcal{O}_K$ is multiplicatively generated by its prime elements. Above any prime number in $\mathbb{Z}$ there is a finite number of prime elements of $\mathcal{O}_K$. Therefore, the cardinality of the prime element of $\mathcal{O}_K$ is the same as $\mathbb{N}$. That is also the cardinality of the prime number of $\mathbb{Z}$. We can therefore create a multiplicative bijection $\varphi$ from $K$ to $\mathbb{Q}$ extending by multiplicativity a bijection from a complete set of prime elements of $ \mathbb{Z}[ \frac{ 1+ \sqrt{-19}}{2}]$, distinct up to units, into a complete set of prime numbers of $\mathbb{Z}$ up to units. Then $(\mathbb{Q}, +_\varphi , \cdot ) \simeq (\mathbb{Q}( \sqrt{-19} ) , + , \cdot)$. We note that we cannot do such a construction with $\mathbb{Q}(i)$ since the group of units has order $4$. Indeed, by Lemma \ref{Zendo}, the primitive $4th$ roots of unity go to the primitive $4th$ roots of unity via a multiplicative isomorphism but there is no primitive $4th$ root of unity in $\mathbb{Q}$. 
\end{exam}

\section{Characterizing additive structures on a fixed scalar group}\label{add_struct}

Let $(F,  \cdot)$ be a scalar group. In order to understand all the near-vector spaces over the scalar group $(F, \cdot)$, we need to understand all the additive binary operations, $\boxplus$, such that $(F, \boxplus,  \cdot)$ is a near-field. 
We start with the following lemma indicating some commonality between left near-fields over the same scalar group $(F, \cdot)$. This also justifies the definition of a scalar group. The results in the lemma are not new, but we add them in a specific format that will be very useful for the rest of the discussion.
\begin{lemm}\label{additive1}
Let $(F, \cdot)$ be a scalar group and $(F, \boxplus, \cdot)$ a left-near field. 
Then $0$ is the zero element with respect to $\boxplus$. Given $\alpha\in F$, the additive inverse of $\alpha$ with respect to $\boxplus$ is either $-\alpha$ or $\alpha$, in which case $1 \boxplus 1 = 0$. Moreover, for all $\alpha, \beta \in F$, we have $\alpha (-1) = (-1) \alpha$. 
\end{lemm}
\begin{proof} 
Let $(F, \cdot)$ be a scalar group and $(F, \boxplus, \cdot)$ a left near-field. We suppose by contradiction that the additive identity in $(F, \boxplus, \cdot)$, that we denote $0'$, is not equal to $0$ in $(F, \cdot)$. Then $0'$ is invertible with respect to $\cdot$ by definition of the scalar group $(F, \cdot)$ and therefore $0'=1$, which contradicts the definition of a near-field $(F, \boxplus, \cdot)$. \\
Now, let $\alpha \in F$. We prove that the additive inverse $\boxminus \alpha$ of $\alpha$ with respect to $\boxplus$ is either $-\alpha$ or $\alpha$. Since $\alpha \boxplus (\boxminus \alpha) = \alpha ( 1 \boxplus (\boxminus 1))$. It is enough to prove that $- 1$ is the additive inverse of $1$ with respect to $\boxplus$. We first prove that $(\boxminus 1)^2=1$. Indeed, $0 =(\boxminus 1) (1\boxplus (\boxminus 1 ))= (\boxminus 1) \boxplus ( \boxminus 1)^2 $ and we obtain the result adding $1$ to both side of the equation.  Therefore, either $\boxminus 1=1$ or $\boxminus 1=-1$.
Thus, if $\boxminus 1  \neq -1$, we have $\boxminus 1=1$. Thus $1 \boxplus 1 =0$. 
Let $\alpha \in F^*$, we have $(\alpha (-1) \alpha^{-1})^2=1$. Therefore, from the above, we get that either $- \alpha = \alpha$ and that implies that $char(F)=2$. Or $\alpha (-1) = (-1) \alpha$. In both cases, we have proven what we wanted. When $\alpha =0$ the result is trivial. 
\end{proof} 
Given a scalar group, the axioms necessary to define an addition such that the addition and the given multiplication define a left near-field are absorbed by the notion of left near-field addition map. 
\begin{defin} \label{nfam}
Let $(F, \cdot)$ be a scalar group.
 A map $\rho: F \rightarrow F$ is said to be {\sf a left near-field addition map} on $(F, \cdot)$ if for all $\alpha \in F^*$ and $\beta \in F$, \\
$1.$ $\rho (0) =1$. We refer to this property as the identity property of $\rho$.\\
$2.$ $\rho (-1)=0$. We refer to this property as the inverse property of $\rho$.\\
$3.$ $\rho ( \alpha^{-1} )= \alpha^{-1} \rho (\alpha)$. We refer to this property as the abelian property of $\rho$. \\
$4.$ $\rho ( \alpha \rho (\beta)) = \alpha \rho ( \beta \rho ( (\alpha \beta)^{-1}))$,
 when $\alpha , \ \beta \in F^*$. We refer to this property as the associative property of $\rho$.
\end{defin} 

\begin{rem} \label{assoc}
Let $(F, \cdot)$ be a scalar group, $\rho$ be a left near-field addition map, and $ \alpha \in F^* \backslash \{ -1\}$. Then we have $\rho ( \alpha \rho (\alpha^{-1} \beta)) = \rho ( \alpha) \rho (   \rho ( \alpha)^{-1} \beta)$.
\end{rem}

\begin{defin}
Let $(F, \cdot)$ be a scalar group.
Given a left near-field $(F, \boxplus, \cdot)$, we denote $\rho_{\boxplus}$ the map sending $\alpha$ to $ 1\boxplus \alpha$. We prove in Lemma \ref{nfa} (1) that $\rho_{\boxplus}$ is a near-field addition map. \\
We define the operation $\boxplus_\rho$ as the operation defined for all $\alpha, \beta \in F$ such that $\alpha \boxplus_\rho \beta = \alpha \rho (\alpha^{-1} \beta) $ when $\alpha\neq 0$  and $\alpha \boxplus_\rho \beta = \beta$, otherwise. We prove in Lemma \ref{nfa} (2) that $(F, \boxplus_\rho, \cdot)$ is a left near-field. We denote $(F, \boxplus_\rho, \cdot)$ simply as ${}_\rho F$ and refer to it as a $\rho$-near-field. When $\alpha_1, \cdots, \alpha_s\in F$, we denote ${}^\rho \sum_{k=1}^s \alpha_k = \alpha_1 \boxplus_\rho \cdots  \boxplus_\rho \alpha_s$. 
\end{defin} 

We now prove that there is a correspondence between left near-field addition maps and left near-field structures over a scalar group. 
\begin{lemm}\label{nfa}
Let $(F, \cdot)$ be a scalar group. 
\begin{enumerate} 
\item Suppose that $(F, \boxplus , \cdot)$ is a left near-field then $\rho_{\boxplus}$ is a near-field addition map on $(F,\cdot)$
\item Suppose that $\rho$ is a near-field addition map on $(F, \cdot)$ then 
$(F, \boxplus_\rho , \cdot)$ is a left near-field. 
\end{enumerate}
\end{lemm} 
\begin{proof}
Let $(F, \cdot)$ be a scalar group. 
\begin{enumerate}
\item Suppose $(F , \boxplus , \cdot)$ is a left near field. Let $\mathcal{A} : F \times F \rightarrow F$  be the binary operation sending $(\alpha , \beta)$ to $\alpha \boxplus \beta$ and $\rho = \mathcal{A} ( 1,-)$.  We have $\rho(0) =\mathcal{A} ( 1,0)= 1\boxplus 0 = 1$, by Lemma \ref{additive1}.
 We also have that $ \mathcal{A} (1, -1) = \rho(-1) = 0$, again by Lemma \ref{additive1}. Furthermore, the commutative property of $\mathcal{A}$ leads to $\rho(\alpha^{-1}) = \mathcal{A} (1, \alpha^{-1}) = \alpha^{-1}  \mathcal{A} (\alpha, 1) = \alpha^{-1}  \mathcal{A} (1, \alpha) = \alpha^{-1} \rho(\alpha)$, for any $\alpha$ non-zero element of $F$. 
 Given $a, b , c\in F^*$. The associativity property means that $\mathcal{A}(\mathcal{A}(a, b), c) = \mathcal{A}(a, \mathcal{A}(b, c))$. The left-hand side gives
\bea
\mathcal{A}(\mathcal{A}(a, b), c) &=& \mathcal{A}(a\mathcal{A}(1, a^{-1}b), c)
= \mathcal{A}(a \rho(a^{-1}b), c) 
= c \mathcal{A}(c^{-1} a \rho(a^{-1} b), 1) \\
&=& c \mathcal{A}(1, c^{-1}  a\rho(a^{-1} b)) = c \rho(c^{-1} a\rho(a^{-1} b)).
\eea
The right-hand side gives
\bea
\mathcal{A}(a, \mathcal{A}(b, c)) &=& \mathcal{A}(a, b\mathcal{A}(1, b^{-1} c)) = \mathcal{A}(a, b \rho(b^{-1} c)) \\ 
&= & a \mathcal{A}(1, a^{-1} \rho(b^{-1} c)) 
= a \rho(a^{-1} b\rho(b^{-1} c)).
\eea
Thus,\begin{equation}  \rho(c^{-1} a  \rho(a^{-1}  b)=c^{-1}  a \rho(a^{-1} b \rho(b^{-1} c)).\end{equation} Setting $\alpha = c^{-1} a$ and $\beta = a^{-1} b$, we have $\alpha \beta = c^{-1} b$ and thus $(\alpha \beta)^{-1} = b^{-1}  c$. 
 Substituting these into Equation (1) we obtain the result $\rho ( \alpha \rho (\beta)) = \alpha \rho ( \beta \rho ( (\alpha \beta)^{-1})$. \\ 
\item Suppose that $\rho$ is a near-field addition map. 
It is clear that $F$ is closed under the operation $\boxplus_\rho$. 
By definition of a left near-field addition map we know that $\rho$ has the properties $\rho (0) =1$, $\rho (-1)=0$, $\rho ( \alpha^{-1} )= \alpha^{-1} \rho (\alpha)$ and  $\rho ( \alpha \rho (\beta)) = \alpha \rho ( \beta \rho ( (\alpha \beta)^{-1})$, for all $\alpha, \beta \in F^*$.  
For any $\alpha,\ \beta, \ \gamma \in F^*$,\\  
$$\alpha \boxplus_\rho  \beta =  \alpha \rho(\alpha^{-1} \beta )=\alpha \alpha^{-1} \beta \rho(\beta^{-1} \alpha )= \beta \rho(\alpha \beta^{-1} )=  \beta \boxplus_\rho  \alpha$$
by the abelian property of $\rho$. 
Moreover, for any $\alpha \in F^*$, $$\alpha \boxplus_\rho 0  = \alpha \rho(\alpha^{-1} 0) =\alpha \rho(0)=  \alpha.$$ and $0 \boxplus_\rho \alpha =\alpha$, for all $\alpha \in F^*$. The result is true by definition when $\alpha=0$. Thus, 0 is the zero element.  This proves simultaneously that $(F, \boxplus_\rho)$ is abelian and has a $0$ element. Moreover, for any $\alpha \in F^*$, 
$$\alpha \boxplus_\rho (-\alpha) = \alpha \rho(-1)  = \alpha  0 = 0.$$ The result is still valid for $\alpha =0$. This shows that each $\alpha \in F$ has an additive inverse $-\alpha$. \\
Let $a, b \in F^*$ and $c \in F$. We now prove the associativity of $\boxplus_\rho$ 
$$\begin{array}{lll} (a\boxplus_\rho b) \boxplus_\rho c &=& a \rho(a^{-1} b) \boxplus_\rho c  =   c   \boxplus_\rho a \rho(a^{-1} b) = c \rho( c^{-1}  a  \rho(a^{-1}  b))\\
&=& a \rho(a^{-1} b \rho(b^{-1} c))= a\boxplus_\rho (b \boxplus_\rho c)
\end{array} $$
from the associative property of $\rho$ applied to $\alpha = c^{-1} a$ and $\beta = a^{-1} b$. When $a=0$ or $b=0$ the result is clear. 
For the distributivity of $\cdot$ over $\boxplus_\rho$, let $a , \alpha , \beta \in F^*$. We have $ a (\alpha \rho (\alpha^{-1} \beta))= a (\alpha \boxplus_\rho \beta )$ and $ (a \alpha)  \rho (\alpha^{-1} \beta) = a \alpha \boxplus_\rho a \beta $. We get $a (\alpha \boxplus_\rho \beta )= a \alpha \boxplus _\rho a \beta$. When either $a=0$ or $\alpha=0$ and $\beta=0$, the result is clear. This concludes the proof that $F$ is a left near-field.
\end{enumerate}
\end{proof}
\begin{rem} \label{rho0} Let $(F, \cdot)$ be a scalar group and $\rho$ be a near-field addition map. 
\begin{enumerate} 
\item We prove that $\rho^n(0) \rho ( \rho^n(0)^{-1} \rho^m(0))=\rho^{n+m} (0) $, for all $n, \ m \in \mathbb{N}$. \\
For $n =1$, the equality is clear since $\rho(0)=1$. Suppose that the equality is true for some $n \in \mathbb{N}$, $\rho^n(0) \rho ( \rho^n(0)^{-1} \rho^m(0))=\rho^{n+m} (0) $. We prove the result for $n+1$. By the induction assumption, we have 
$$\begin{array}{lll} \rho^{n+1+m} (0)&=& \rho( \rho^{n+m} (0))= \rho (\rho^n(0) \rho ( \rho^n(0)^{-1} \rho^m(0))) \\
&=&  \rho^{n+1} (0) \rho ( \rho^{n+1}(0)^{-1} \rho^m(0))),
\end{array}$$
by Remark \ref{assoc}.\\
This proves the result by induction. 
\item For all $n \in \mathbb{N}$, $\rho^n(0) \rho ( - \rho^n(0)^{-1} \rho^n(0) )=\rho^n(0) \rho ( - 1 )=0$
\end{enumerate}
\end{rem}
From the definition of a near-field map, we naturally obtain that they are bijective maps. 
\begin{lemm}
A near-field addition map is bijective. More precisely, if $\rho$ is a near-field addition map then its inverse is the map sending $\alpha$ to $-\rho(-\alpha)$. In particular, $\rho (\alpha ) \neq 0$ for any $\alpha \neq -1$.
\end{lemm}
\begin{proof}
Let $\rho$ be a near-field addition map. We prove that $\sigma : F \rightarrow F$ sending $\alpha$ to $- \rho (-\alpha)$ is the inverse of $\rho$. Let $ a \in F^*$. Taking $\alpha= a^{-1} $ and $\beta=-1$, the property $\rho ( \alpha \rho (\beta)) = \alpha \rho ( \beta \rho ( (\alpha \beta)^{-1})$ becomes $\rho ( a^{-1} \rho (-1)) = a^{-1} \rho (-1 \rho ( - a ))$.  Since $\rho ( a^{-1} \rho (-1))  = \rho ( 0)=1$ and $a^{-1}  \rho (-1 \rho ( - a ) )= a^{-1} \rho (\sigma (  a ) )$, we obtain $\rho (\sigma (  a ) )=a$. For $a=0$, we have $\rho (-1 \rho (0 ) )=\rho(-1)=0$, so that $\rho (\sigma (  a ) ) = a$. To prove that $\sigma (\rho (  a ) ) = a$, we take $\alpha = - a^{-1}$ and $\beta=-1$. The property $\rho ( \alpha \rho (\beta)) = \alpha \rho ( \beta \rho ( (\alpha \beta)^{-1})$ becomes $\rho ( -a^{-1} \rho (-1)) = - a^{-1} \rho (- \rho ( a ))$ so that, as before, we get $1 = - a^{-1} \rho (- \rho ( a ))$. For $a =0$, we have 
$\sigma (\rho (  a ) )= - \rho (- \rho ( 0 ))=-\rho(-1)=0$. That is $\sigma (\rho ( a ) ) = a$. This proves that $\rho$ is a bijection. 
\end{proof} 

\begin{rem}
For every $\alpha \in F \backslash \{ -1\}$, we have $\rho (\alpha)\in F^*$.  
\end{rem}

We now define the additive structure on a scalar group induced by a addition near-field map. 
 We start with some basic properties of repeated addition. 
 \begin{lemm} \label{rho0} Let $\alpha, \beta \in F$ and $n \in\mathbb{N}$. We have $ {}^\rho \sum_{k=1}^n \alpha =  \alpha \rho^{n}(0) $ and $ {}^\rho \sum_{k=1}^n \alpha  \,\boxplus_\rho \,{}^\rho \sum_{k=1}^n \beta =  (\alpha \boxplus_\rho \beta) \rho^{n}(0) $. 
  \end{lemm} 
 \begin{proof} 
 Let $\alpha \in F$. We prove that $ {}^\rho \sum_{k=1}^n \alpha =  \alpha \rho^{n}(0) $.
 When $\alpha =0$, the result is trivial. We prove the result by induction on $n$. Let $\alpha \in F^*$. For $n=1$, the result is trivial. We suppose that $ {}^\rho \sum_{k=1}^n \alpha = \alpha \rho^{n}(0)$ is true for some $n$. We prove that it remains true for $n+1$. Using the induction assumption we obtain: 
 $$ {}^\rho \sum_{k=1}^{n+1}  \alpha =  \alpha \boxplus_\rho {}^\rho \sum_{k=1}^{n}  \alpha  =  \alpha \boxplus_\rho   \alpha \rho^{n}(0)= \alpha \rho ( \alpha^{-1} \alpha \rho^{n}(0) )= \alpha \rho^{n+1}(0).$$
 Moreover,  let $\alpha, \beta \in F$. We have 
 $$\begin{array}{ccc} {}^\rho \sum_{k=1}^n \alpha  \boxplus_\rho {}^\rho \sum_{k=1}^n \beta  &=&  {}^\rho \sum_{k=1}^n (\alpha  \boxplus_\rho \beta) = (\alpha  \boxplus_\rho \beta)  \rho^{n}(0)
 \end{array}$$ 
 above we have used the abelian property of $ \boxplus_\rho$.
 \end{proof} 
We next define the notion of a characteristic map. This map will permit us to define the characteristic of a near-field.
\begin{defin} \label{rhocha}
We define the $\rho$-characteristic map, denoted $\chi_\rho$, to be the map 
$$ \begin{array}{llll} \chi_\rho: & \mathbb{Z} &\rightarrow &F\\
& n &\mapsto &  \operatorname{sgn} (n)\rho^{|n|}(0). \end{array}
$$
where $\rho^{|n|}$ denotes $\rho$ composite with itself $|n|$ times when $n\neq 0$ and we set $\rho^{0}= \operatorname{id}$. We also denote $\mathcal{C}_\rho$ the image of the map $\chi_\rho$.
\end{defin}
We obtain as, expected and is known, that the characteristic of a near-field is either $0$ or a prime number. The characteristic map defines an embedding of a prime field onto any near-field. 
\begin{lemm}(Definition) \label{charac}
 Let $(F, \cdot)$ be a scalar group and $\rho$ be a near-field addition map. Then $\chi_\rho$ is a ring homomorphism from $(\mathbb{Z}, +, \cdot )$ to $(F, \boxplus_\rho, \cdot)$. $(\mathcal{C}_\rho , \boxplus_\rho)$ is a cyclic group isomorphic either to $\mathbb{Z}$, when $\chi_\rho$ is one-to-one, or $\mathbb{F}_p$ for some $p$ prime otherwise. 
 When $
 \chi_\rho$ is one-to-one, $ \chi_\rho$ naturally induces a field morphism from $( \mathbb{Q}, +, \cdot)$ to $(F, \boxplus_\rho, \cdot)$
 $$\begin{array}{llll} \widetilde{\chi_\rho} : & \mathbb{Q} & \rightarrow  & F\\ & \frac{n}{m} & \mapsto & \operatorname{sgn} (nm)\rho^{|n|}(0) (\rho^{|m|}(0))^{-1}
 \end{array} $$ 
  localizing at the prime ideal $(0)$ of $\mathbb{Z}$. 
 Otherwise, $ \chi_\rho$ naturally induces a field morphism from $( \mathbb{F}_p, +, \cdot)$ to $(F, \boxplus_\rho, \cdot)$
  $$\begin{array}{llll} \widetilde{\chi_\rho} : & \mathbb{F}_p & \rightarrow  & F\\ & [n]_p & \mapsto & \operatorname{sgn} (n)\rho^{|n|}(0)
 \end{array} $$
using the first isomorphism theorem.  The {\sf characteristic $p$} of $\rho$ denoted $char(F)$ is $0$ when $\chi_\rho$ is one-to-one and $p$ when $ker(\chi_\rho)=p \mathbb{Z}$. We denote $F_p$ to be the field $\mathbb{Q}$ when $p=char(F)=0$ and $\mathbb{F}_p$ when $char(F)=p$. We denote $\widetilde{C_p}= \widetilde{\chi_\rho} ( F_p)$. In particular, when $F$ is finite, $|F|= p^n$ for some $p$ prime and $n \in \mathbb{N}$. A left near-field such that $\widetilde{\chi_\rho}$ is an isomorphism is referred to as {\sf a prime field.}  
\end{lemm}
\begin{proof}Let $n,m \in \mathbb{N}$.
By Remark \ref{rho0}, $\chi_\rho(n+m)=\chi_\rho(n) \boxplus_\rho \chi_\rho(m)$ and $\chi_\rho(-n-m)=-\chi_\rho(n) \boxplus_\rho -\chi_\rho(m)$, by distributivity of $\cdot$ on $+_\rho$. We now assume that $\operatorname{sgn} ( n)=-1$, $\operatorname{sgn} ( m)=1$ and $n\leq m$. We have
$$\begin{array}{lll} \chi_\rho(n) \boxplus_\rho \chi_\rho(m) &=& \chi_\rho(n) \boxplus_\rho (\chi_\rho(-n) \boxplus_\rho  \chi_\rho(m-n)) \ \ \text{ by the above } \\
&=& (\chi_\rho(n) \boxplus_\rho \chi_\rho(-n) )\boxplus_\rho  \chi_\rho(m-n)\\
&=& 0\boxplus_\rho  \chi_\rho(m-n) \ \ \text{ by Remark \ref{assoc}}\\
&=& \chi_\rho(m-n)
\end{array}
$$
The other cases are proven using the distributivity of $\cdot$ on $+_\rho$ and exchanging the role of $n$ and $m$. That concludes the proof that $\chi_\rho$ is an additive homomorphism. Let $n \in \mathbb{N}$ and $m \in \mathbb{Z}$. We have $\chi_\rho ( nm ) =\chi_\rho ( \sum_{k=1}^n m )= {}^\rho \sum_{k=1}^n \chi_\rho ( m)= \chi_\rho ( m) \chi_\rho(n)  $, by Lemma \ref{rho0}. When $nm < 0$, we have $$\chi_\rho ( nm )= \chi_\rho ( -|nm| )= -\chi_\rho ( |nm| )=  - \chi_\rho ( |n|) \chi_\rho (| m| )= \chi_\rho ( n) \chi_\rho ( m ).$$ 
Suppose that $ \chi_\rho$ is not one-to-one. We prove that the kernel of $ \chi_\rho$ is $p \mathbb{Z}$ where $p$ is a prime number. Since $\chi_\rho$ is a ring homomorphism we know that its kernel is of the form $n \mathbb{Z}$ for some $n \in \mathbb{N}$. We suppose by contradiction that $n$ is not prime. That is there is $a, b \in \mathbb{N} \backslash \{ 1 \}$ such that $n = a b$. Moreover, since $n$ is in the kernel $\chi_\rho ( n) = \chi_\rho (ab) = \chi_\rho (a)  \chi_\rho (b)= 0$. So that either $ \chi_\rho (a) =0$ or $ \chi_\rho (b)=0$. This is a contradiction since $a$ and $b$ are proper divisors of $n$ and therefore cannot belong to $n \mathbb{Z}$.
When $\chi_\rho$ is one-to-one, the field homomorphism  $\widetilde{\chi_\rho}$ results from the universal property of the localization. Indeed, when $\chi_\rho$ is one-to-one any non-zero elements of $\mathbb{Z}$ are sent to a non-zero element of $F$ via $\chi_\rho$, therefore, is invertible. Otherwise, $\widetilde{\chi_\rho}$ is the morphism induced by the first isomorphism theorem. 
\end{proof}

The embedding of the prime near-field onto a near field of characteristic $p$ gives a $F_p$-vector space structure on $F$. We can deduce from this the uniqueness of the additive structure up to isomorphism for finite fields. 
\begin{corollary} \label{uniqueup} Let $(F, \cdot)$ be a scalar group. Under the notation of Lemma \ref{charac}.
 $\widetilde{\chi_\rho}$ induces a $F_p$-vector space structure on $F$. In particular, any element of $\widetilde{C_p}$ distributes on any element of $F$ and $\widetilde{C_p}^*$ is a commutative multiplicative subgroup of $F^*$. When $F$ is a finite near-field then $F$ has a unique additive structure $\boxplus$ on $F$ up to isomorphism such that $(F, \boxplus, \cdot)$ is a near-field.  
 \end{corollary}
 \begin{proof}
Suppose $char(F)=p$. We now prove that $\widetilde{\chi_\rho}$ induces a $F_p$-vector space structure on $F$. We define the scalar multiplication by $\alpha \cdot_\chi a= \alpha \cdot \chi_\rho (a)$ for all $a\in F_p$ and $ \alpha \in F$. 
 The distributivity on the left of $\cdot_\chi$ onto $\boxplus_\rho$ is guaranteed by the fact that $F$ is a left near-field, while the distributivity on the right of $\cdot_\chi$ onto $\boxplus_\rho$ follows from Lemma \ref{rho0}.
 We can then easily deduce that $F$ is a $F_p$-vector space with respect to $\boxplus_\rho$ and $\cdot$.
This implies that $\widetilde{C_p}$ also distributes on the right. 
When $F$ is finite near-field, since $F$ is a $F_p$-vector space, we have $F \simeq  \mathbb{F}_p^n$ where $p = Char(F)$ and $n$ is the integer such that $|F|=p^n$. This proves the uniqueness of the additive map up to additive isomorphism, in this case.
\end{proof} 

In the next lemma, we investigate when a bijection $\sigma$ induces a near-field structure on a given scalar group. From the result, as it stands, we see that there is an underlying quasi-multiplicative automorphism on $\widetilde{C_p}$ where $p= char(F)$. It is not clear if there could be a field structure induced by a bijection that is not isomorphic to the initial field structure. 

 \begin{theo} \label{bij+}
  Let $(F, + , \cdot)$ be a left near-field of characteristic $p$ and $\sigma$ be a bijection of $F$.
  Then $(F, +_\sigma, \cdot)$ is a left near-field if and only if $\alpha (\beta +_\sigma \gamma) =   \alpha \beta +_\sigma \alpha \gamma
$, for any $\alpha, \beta, \gamma \in F$. In particular, when $(F, +_\sigma, \cdot)$ is a left near-field then $\sigma$ is a $\mathbb{Z}$-endobijection such that $\sigma ( \alpha ( (  \beta a )_\sigma ))  = \sigma ( \alpha \beta_\sigma) a$, for all $a \in \widetilde{C_{p}} $ and $\alpha \in F$. 
We have for all $\alpha \in \widetilde{C_{p}}$ and $\beta \in F$, $\sigma^{-1} ( \beta \alpha  ) = \sigma^{-1} ( \beta ) \sigma^{-1} ( \sigma ( 1) \alpha)$. Setting $\widetilde{\sigma^{-1}}$ to be the map sending $\alpha \in F$ to $\widetilde{\sigma^{-1}} ( \alpha ) = \sigma^{-1}( \sigma (1) \alpha)$, we have for all $\alpha \in \widetilde{C_p}$ and $\beta \in F$, $\widetilde{\sigma^{-1}} ( \beta \alpha )= \widetilde{\sigma^{-1}} ( \beta  ) \widetilde{\sigma^{-1}} ( \alpha )$. In particular, $\sigma|_{\widetilde{C_{p}}}$ is a quasi-multiplicative endobijection of $\widetilde{C_{p}}$ associated with the multiplicative morphism $\widetilde{\sigma}$ defined by $\widetilde{\sigma}( \alpha )= \sigma ( \alpha ) \sigma (1)^{-1}$ for any $\alpha \in \widetilde{C_p}$ and $\sigma (1)$. In particular, when $(F, + , \cdot)$ is $(\mathbb{Q}, +, \cdot)$ the canonical rational field, $(\mathbb{Q}, +_\sigma, \cdot)$ is isomorphic to $(\mathbb{Q}, +, \cdot)$. 
 \end{theo} 
 \begin{proof}
Since $(F, \cdot)$ is a scalar group and $(F, +_\rho)$ is an abelian group, $(F, +_\sigma, \cdot)$ is a left near-field if and only if $\alpha (\beta +_\sigma \gamma) =   \alpha \beta +_\sigma \alpha \gamma$, for any $\alpha, \beta, \gamma \in F$. 
By Lemma \ref{additive1}, we know that $\sigma (\pm 1)=\pm \sigma(1)$,  $\sigma(0)= 0$ and $\sigma (-\alpha)=- \sigma(\alpha)$, for all $\alpha \in F$. That is $\sigma$ is a endobijection of $F$. We set $\rho:=\rho_{+_\sigma}$. Let $n \in \mathbb{N}$. From the associativity of $+_\sigma$, we obtain for any $\alpha, \beta\in F$, 
$$ \begin{array}{lrll} & \alpha ( \sigma^{-1} (\sum_{i=1}^n \sigma (\beta_\sigma )) ) &=&  \sigma^{-1} (\sum_{i=1}^n \sigma ( \alpha \beta_\sigma) ) \\
\Leftrightarrow  & \alpha ( \sigma^{-1} (\beta \rho^n (0)) )  &=&  \sigma^{-1} (  \sigma ( \alpha \beta_\sigma) \rho^n (0) )  \\
\Leftrightarrow  & \sigma ( \alpha ( (\beta \rho^n (0))_\sigma ))  &=&  \sigma ( \alpha \beta_\sigma) \rho^n (0)\\
\Leftrightarrow  & \sigma ( \alpha ( (\beta\chi_\rho ( n ))_\sigma ))  &=&  \sigma ( \alpha \beta_\sigma) \chi_\rho ( n ).
\end{array} $$
When $n=0$, then $\sigma ( \alpha ( \sigma^{-1} (n) ))  = \sigma ( \alpha \beta_\sigma) n  =0$. 
When $n \in \mathbb{Z}$, $n<0$ then we also have $\sigma ( \alpha ( (\beta\chi_\rho ( n ))_\sigma ))  =  \sigma ( \alpha \beta_\sigma) \chi_\rho ( n ), $ since $\chi_\rho ( n )= - \chi_\rho ( |n| )$ and $\sigma$ is a $\mathbb{Z}$-endomap.
Therefore, for all $n \in \mathbb{Z}$, 
$$\sigma ( \alpha ( (  \beta \chi_\rho ( n )  )_\sigma ))  =  \sigma ( \alpha \beta_\sigma) \chi_\rho ( n ) .$$
When $Char(F)=0$. Let $n, \ m  \in \mathbb{Z}$ such that $\chi_\rho ( n ) \in F^*$. We have 
$$ \begin{array}{lrrl} 
& \sigma ( \alpha ( (\beta \chi_\rho ( n )  )_\sigma)  &=& \sigma ( \alpha \beta_\sigma) \chi_\rho ( n )    \\
& \sigma ( \alpha ( (\beta \chi_\rho ( n )  (\chi_\rho ( m ))^{-1} \chi_\rho ( m ) )_\sigma)  &=&  \sigma ( \alpha \beta_\sigma) \chi_\rho ( n )  (\chi_\rho ( m ))^{-1}\chi_\rho ( m ) \\
 \Leftrightarrow  & \sigma ( \alpha ( (\beta \chi_\rho ( n )  (\chi_\rho ( m ))^{-1}  )_\sigma)\chi_\rho ( m ) &=&  \sigma ( \alpha \beta_\sigma) \chi_\rho ( n )  (\chi_\rho ( m ))^{-1}\chi_\rho ( m )\\
  \Leftrightarrow  & \sigma ( \alpha ( (\beta \chi_\rho ( n )  (\chi_\rho ( m ))^{-1}  )_\sigma)&=&  \sigma ( \alpha \beta_\sigma) \chi_\rho ( n )  (\chi_\rho ( m ))^{-1}\\
    \Leftrightarrow & \sigma ( \alpha ( (\beta \widetilde {\chi_\rho} \left( \frac{n}{m} \right)  )_\sigma)&=&  \sigma ( \alpha \beta_\sigma)  \widetilde {\chi_\rho} \left( \frac{n}{m} \right).
\end{array} $$
We take $\alpha = b^{-1}_\sigma$, and $\beta = b$, and $a \in \widetilde{C_p}$, we get 
$$\begin{array}{lrrl} & \sigma (  b^{-1}_\sigma (b a )_\sigma )  &=&   \sigma (b_\sigma^{-1} b_\sigma)  a\\
   \Leftrightarrow  &  b^{-1}_\sigma ( ( b a )_\sigma ) &=& \sigma^{-1} (\sigma ( 1) a )\\
      \Leftrightarrow & \sigma^{-1} ( b a ) &=& \sigma^{-1} ( b) \sigma^{-1} (\sigma(1) a).
\end{array} $$

We define the map $\widetilde{\sigma^{-1}}$ as the map sending $\alpha$ to $\widetilde{\sigma^{-1}} ( \alpha )= \sigma^{-1}(\sigma (1) \alpha )$ for any $\alpha \in F$.
We have for any $a \in \widetilde{C_p}$ and $b \in F$. 
$$\widetilde{\sigma^{-1}} (  b a )=\sigma^{-1} (  \sigma (1) b  a ) = \sigma^{-1} (\sigma ( 1)  b) \sigma^{-1} ( \sigma ( 1) a)= \widetilde{\sigma^{-1}} (  b ) \widetilde{\sigma^{-1}} (  a )$$

Also, by Lemma \ref{qmc}, we know that $\sigma^{-1}$ is a quasi-multiplicative endobijection of $\widetilde{C_p}$ associated with the multiplicative morphism $\widetilde{\sigma^{-1}}$ defined by $\widetilde{\sigma^{-1}} ( \alpha )= \sigma^{-1}( \alpha \sigma (1))$ for any $\alpha \in F$ and $\sigma^{-1}( 1)$. Therefore, by Lemma \ref{eqm}, $\sigma$ is a quasi-multiplicative endobijection of $\widetilde{C_p}$ associated with the multiplicative morphism $\widetilde{\sigma}$ defined by $\widetilde{\sigma}( \alpha )= \sigma ( \alpha ) \sigma (1)^{-1}$ for any $\alpha \in \widetilde{C_p}$ and $\sigma (1)$.

When $(F, + , \cdot)$ is $(\mathbb{Q}, +, \cdot)$ the canonical rational field, $(\mathbb{Q}, +_\sigma, \cdot)$ is isomorphic to $(\mathbb{Q}, +,\cdot)$, by Corollary \ref{dvr}. 
\end{proof} 

Due to the work done in this section we can now describe all the possible additive structures over a finite field. 

Given $(F, \boxplus, \cdot)$ a finite field.  Then there is a field isomorphism between $(F, \boxplus, \cdot)$ and $(\mathbb{F}_{p^n} , + , \cdot)$ where $+$ is the usual addition in $\mathbb{F}_{p^n}$.  In particular, $\mathbb{F}_{p^n}$ has a unique additive structure up to field isomorphism over the scalar group $(\mathbb{F}_{p^n}, \cdot)$ where $\cdot$ is the usual multiplication. The additions on $\mathbb{F}_{p^n}$ are given by $ ( \alpha^{a} + \beta^{a})^{a^{-1_n}}$ where $a \in U_{p^n-1}$ and $a^{-1_n}$ is a representative for the multiplicative inverse of $[a]_{p^n-1}$ in the multiplicative group $U_{p^n-1}$ of $\mathbb{Z}/(p^n-1)\mathbb{Z}$. Indeed, the lemma is a result of the uniqueness of finite fields up to isomorphism and Lemma \ref{ringisom}. The structure of the multiplicative automorphism of $\mathbb{F}_{p^n}$ is given by the fact the $\mathbb{F}_{p^n}^*$ is a cyclic group of order $p^n -1$. 

We have seen that depending on the field structure put on $\mathbb{Q}$, $\mathbb{Q}$ might not be a prime field (see Example \ref{exq}).  More precisely, given a field structure on $\mathbb{Q}$, $(\mathbb{Q}, \boxplus, \cdot)$, since we have proven that $\mathbb{Q}$ with respect to $\boxplus$ is a vector space over $\mathbb{Q}$. We have that $(\mathbb{Q}, \boxplus, \cdot)$ is isomorphic to a field extension of $\mathbb{Q}$ say $(K , + , \cdot)$. In order for such isomorphism to exist we need a multiplicative bijection between $K$ and $\mathbb{Q}$, by Lemma \ref{ringisom}. The sufficient conditions for the isomorphism to exist are that: the group of units of the ring of integers of $K$ is $\{ \pm 1\}$ the group of units of $\mathbb{Z}$, the ring of integers of $K$, $\mathcal{O}_K$, is a unique factorization domain and the cardinality of the prime numbers of $K$ up to unit is the same as the cardinality of $\mathbb{N}$. Indeed, $\mathbb{Z}$ is generated freely multiplicatively by the set of prime numbers $\mathbb{P}$, and this set has the same cardinality as $\mathbb{N}$. For any ring of integers with those properties, there is an addition, $\boxplus=+_\phi$ for some $\phi: (K, \cdot)  \rightarrow (F, \cdot)$ multiplicative bijection, on $\mathbb{Q}$ such that $(\mathbb{Q}, \boxplus, \cdot)$ is isomorphic to a number field $(K, +, \cdot)$. $\phi$ is obtained by extending multiplicatively a bijection from $\mathbb{P}$ into a complete set of prime elements of $ \mathcal{O}_K$ distinct up to units. The prime subfield $(\mathbb{Q}, +, \cdot) $ of $(\mathbb{Q}, \boxplus, \cdot) $ is multiplicatively generated by the image of inert primes in $K/\mathbb{Q}$ and the inverse image via $\phi$ of the norm $K$-prime integers non-inert.   
Moreover, we have $\varphi = (-)^a$ where $a \in\{ \pm 1\}$ are the only continuous multiplicative endobijection of $\mathbb{Q}$ for the topology on $\mathbb{Q}$ induced by the classical topology on $\mathbb{R}$. 

Let $\rho$ be a near-field addition map on $\mathbb{R}$. If we suppose that $\widetilde{\chi_\rho}$ can be extended by continuity to the completion of $\mathbb{Q}$, $\mathbb{R}$ and the resulting map $\widehat{\chi_\rho} : \mathbb{R} \rightarrow \mathbb{R}$ is a embedding of $\mathbb{R}$. If we suppose that $\widehat{\chi_\rho}$ is an isomorphism, then $\alpha \boxplus_\rho \beta =\alpha +_\varphi \beta $ 
where $\varphi$ is a multiplicative endobijection of $\mathbb{R}$ (see Definition \ref{R} and \cite[Theorem 3.7]{MarquesBoonz2022A}).

Let $\rho$ be a near-field addition map on $\mathbb{C}$. If we suppose that $\widetilde{\chi_\rho}$ can be extended by continuity to the completion of $\mathbb{Q}$, $\mathbb{R}$ to a map $\widehat{\chi_\rho} : \mathbb{R} \rightarrow \mathbb{C}$. We note that since $\widehat{\chi_\rho}$ is multiplicative and $(\mathbb{R}, \cdot)$ has no element of order $4$, then $\widehat{\chi_\rho}$ cannot be an isomorphism. Suppose that the evaluation map $\overline{\chi_\rho} : \mathbb{R}[x] \rightarrow \mathbb{C}$ sending $\sum_{k=0}^n a_k x^k$ to $\sum_{k=0}^n \widehat{\chi_\rho} (a_k) i^k$ where $a_k \in \mathbb{R}$, for all $k \in \{ 0,\cdots , n\}$ is surjective and $ker(\overline{\chi_\rho})$ is generated by an irreducible degree $2$ polynomial over $\mathbb{R}$, then $\alpha \boxplus_\rho \beta =\alpha +_\varphi \beta $  
where $\varphi$ is a multiplicative endbijection of $\mathbb{C}$ (see Definition \ref{R} and \cite[Theorem 3.12]{MarquesBoonz2022A}). Due to to Lemma \ref{ringisom}, to prove the above it suffices to prove that given an irreducible monic polynomial $x^2 +bx +c$ over $\mathbb{R}$, we have a field isomorphism $(\mathbb{R}[x] /(x^2 +bx +c), + , \cdot)  \simeq (\mathbb{C}, + , \cdot)$ where $+$ and $\cdot$ are the usual operations on $\mathbb{R}[x] /(x^2 +bx +c)$ and $\mathbb{C}$.  As usual completing the square, we obtain $\left( x+\frac{b}{2}\right)^2 - \frac{b^2}{4}+c =0$. That is $\left( x+\frac{b}{2}\right)^2 = \frac{b^2-4c}{4}$. Since the polynomial is irreducible over $\mathbb{R}$, we have $-\frac{b^2-4c}{4}>0 $ and $\left( x+\frac{b}{2}\right)^2 = (\sqrt{ -\frac{b^2-4c}{4}} i)^2  $. Equivalently, setting $X=\frac{1}{\sqrt{ -\frac{b^2-4c}{4}}} \left( x+\frac{b}{2}\right)$, we get $X^2 +1=0$. We obtain a field isomorphism from $(\mathbb{R}[x] /(x^2 +bx +c), + , \cdot)$ to $(\mathbb{R}[X] /(X^2+1), + , \cdot)$ sending $x$ to $\sqrt{ -\frac{b^2-4c}{4}} x-\frac{b}{2}$ and by definition, $(\mathbb{R}[X] /(X^2+1), + , \cdot) \simeq (\mathbb{C}, + , \cdot)$.

\end{document}